\documentclass{tran-l_for_arXiv}

\usepackage{graphicx}
\usepackage{amssymb}
\usepackage{amscd}
\usepackage{amsmath}
\usepackage{amsthm}
\usepackage[noend]{algpseudocode}
\usepackage{algorithm}
\usepackage{enumerate}
\usepackage[all]{xy}
\usepackage{comment}
\usepackage{subfigure}

\newtheorem{thm}{Theorem}[section]
\newtheorem{lem}[thm]{Lemma}
\newtheorem{defn}[thm]{Definition}
\newtheorem{cor}[thm]{Corollary}
\newtheorem{prop}[thm]{Proposition}
\newtheorem{ex}[thm]{Example}
\newtheorem{rem}[thm]{Remark}

\newcommand{\F}{{\mathbb{F}}}
\newcommand{\N}{{\mathbb{N}}}
\newcommand{\R}{{\mathbb{R}}}

\newcommand{\Z}{{\mathbb{Z}}}

\newcommand{\cC}{{\mathcal C}}

\newcommand{\cK}{{\mathcal K}}

\newcommand{\cS}{{\mathcal S}}

\newcommand{\cX}{{\mathcal X}}

\def\vector#1{\mbox{\boldmath $#1$}}

\newcommand{\rank}{\mathop{\rm rank}}
\newcommand{\image}{\mathop{\rm im}}
\newcommand{\kernel}{\mathop{\rm ker}}

\newcommand{\llangle}{\langle\!\langle}
\newcommand{\rrangle}{\rangle\!\rangle}

\newcommand{\positivereal}{\mathbb{R}_{\geq 0}}

\def\positivereal{\mathbb{R}_{\ge 0}}
\def\extendedpositivereal{\overline{\mathbb{R}}_{\ge 0}}

\def\llangle{\langle \! \langle}
\def\rrangle{\rangle \! \rangle}
\def\reduce{\mathcal{R}_d}
\def\shadow{\mathcal{S}_d}
\def\cnd{\mathcal{C}_{n}^{(d)}}
\def\cnmd{\mathcal{C}_{n,m}^{(d)}}

\def\n{{\mathbb N}}
\def\a{{\mathcal A}}

\numberwithin{equation}{section}
\makeatletter
\def\filename{\texttt{\jobname.tex}} 
\makeatother

\def\P{{\mathbb P}} 
\def\F{{\mathcal F}} 
\def\E{{\mathbb E}} 
\def\r{{\mathbb R}} 
\def\la{\lambda}
\def\x{{\bf x}} 
\def\y{{\bf y}} 
\def\diag{{\rm diag}} 
\def\lexmin{{\rm lexmin}} 
\def\CC{\mathcal{C}}
\def\K{\mathcal{K}}

\newcommand{\wt}{{\rm wt}}

\newcommand\lref[1]{Lemma~\ref{#1}}

\newcommand\pref[1]{Proposition~\ref{#1}}

\newcommand\cref[1]{Corollary~\ref{#1}}

\begin{document}

\title[Minimum spanning acycle and lifetime of persistent homology]{Minimum spanning acycle and lifetime of persistent homology 
in the Linial-Meshulam process}


%    author one information
\author{Yasuaki Hiraoka}
\address{Institute of Mathematics for Industry, Kyushu University,
744, Motooka, Nishi-ku, Fukuoka, 819-0395, Japan}
\curraddr{}
\email{hiraoka@imi.kyushu-u.ac.jp}
\thanks{}

%    author two information
\author{Tomoyuki Shirai}
\address{Institute of Mathematics for Industry, Kyushu University,
744, Motooka, Nishi-ku, Fukuoka, 819-0395, Japan}
\curraddr{}
\email{shirai@imi.kyushu-u.ac.jp}
\thanks{}

%\addeditor{H}
%\addeditor{T}

%    \subjclass is required.

\date{}

\dedicatory{}

%    Abstract is required.
\begin{abstract}
This paper studies a higher dimensional generalization of Frieze's $\zeta(3)$-limit theorem in the
Erd\"os-R\'enyi graph process.  Frieze's theorem states that the
expected weight of the minimum spanning tree converges to $\zeta(3)$ as
the number of vertices goes to infinity.  In this paper, we study the
$d$-Linial-Meshulam process as a model for random simplicial complexes,
where $d=1$ corresponds to the Erd\"os-R\'enyi graph process.  
First, we define spanning acycles as a higher dimensional analogue of spanning trees, and connect its minimum weight to persistent homology. 
Then, our main result shows that 
the expected weight of the minimum spanning acycle 
behaves in $O(n^{d-1})$.
\end{abstract}

\maketitle


\section{Introduction}\label{sec:intro}

%Here it is more appropriate to consider the time dependent version 
%of the Erd\"os-R\'enyi graph by considering $p$ as time parameter taking values in $[0,1]$. 
%We recall the definition of Erd\"os-R\'enyi dynamics. 
Let $K_n = V_n \sqcup E_n$ be the complete graph with $n$ vertices,
where $V_n$ and $E_n$ are the sets of vertices and edges, respectively.   
We assign a uniform random variable $t_e \in [0,1]$ independently for
each edge $e \in E_n$, and define an increasing stochastic 
process of subgraphs of $K_n$ by 
\begin{equation}
 K_n(t) = V_n \sqcup \{e \in E_n \mid t_e \le t\}, \quad t \in [0,1]. 
\label{eq:erdosrenyi} 
\end{equation}
%is a increasing sequences of subgraphs of $K_n$ 
This process starts from $V_n$ at time $t=0$ and ends up with
 $K_n$ at time $t=1$.  
%We call it the Erd\"os-R\'enyi graph process. 
It is called the Erd\"os-R\'enyi graph process. 
By definition, $K_n(t)$ is equal in law to the Erd\"os-R\'enyi graph $G(n,t)$, 
which is obtained from $K_n$ by retaining each edge with probability $t$
and deleting it with probability $1-t$ independently
\cite{er}. 
%The Erd\"os-R\'enyi graph (process) has been intensively studied by many
%authors. 
We also note that $K_n(t)$ defines a \textit{random} filtration of
$K_n$ parametrized by $t\in[0,1]$. 

Let $\cS^{(1)}$ be the set of spanning trees in $K_n$, i.e., the trees
in $K_n$ containing all vertices.  
%Note that the spanning trees consist of $n-1$ edges. 
Note that every spanning tree consists of $n-1$ edges.
The minimum spanning tree on $K_n$ is defined 
as the spanning tree $T \in \cS^{(1)}$ 
with the minimum weight $\wt(T) = \sum_{e \in T} t_e$. 
Here, it is worth mentioning Kruskal's algorithm \cite{kruskal}
for finding the minimum 
spanning tree. In Kruskal's algorithm, the weights $\{t_e\}_{e \in E_n}$
are treated as the birth times of edges. We start from the isolated
vertices $V_n$ at time $0$, and then we expose an edge $e$ at time $t_e$
in order. If the edge $e$ does 
not create a cycle, we keep it remained in our graph; otherwise we omit it. 
%We repeat this procedure until the number of accepted edges becomes $n-1$, and finally we obtain the minimum spanning tree $T_{\text{min}}$ and its weight $\wt(T_{\text{min}})$.  
We repeat this procedure until the number of accepted edges becomes $n-1$, and the derived tree will be the minimum spanning tree. 
%\note[H]{",,, $\wt(T_{\text{min}})$ as the sum of birth times of 
%accepted edges, which can be dynamically computed." is removed.}\note[S]{OK.}
%. The left-hand side of (\ref{eq:minwgttree}) is the sum of birth times of 
%accepted edges, which can be dynamically computed. 

Frieze \cite{frieze} shows the following significant result about the weight of the minimum spanning tree.

\vspace{0.3cm}\noindent
{\bf Frieze's $\zeta(3)$-Limit Theorem.}\\
\begin{equation}\label{eq:frieze}
\E[\min_{T \in \cS^{(1)}} \wt(T)] \to \zeta(3)=1.202\cdots
\end{equation}
as $n \to \infty$, where $\zeta(s)$ is Riemann's zeta function. 
\vspace{0.3cm}\\

This limit theorem has been investigated further in several directions, e.g., 
%the central limit theorem around it, extensions to general weight distributions, a large deviations estimate, errors of convergence, and asymptotic expansions. 
a central limit theorem and a tail estimate for the minimum weight, extensions to more general weight distributions and underlying graphs, and asymptotic expansions.
Recent developments can be found in \cite{cfijs} and references therein.
%In the present paper, we will explore yet another way to cultivate. 
In the present paper, we will explore a higher dimensional generalization of this limit theorem.

%This theorem is derived by the deterministic formula
One of the main ingredients in the proof of Frieze's theorem
is the following formula connecting the weight of 
the minimum spanning tree to the integrated Betti number:  
\begin{equation}\label{eq:minwgttree}
%\E[\min_{T \in \cS^{(1)}} \wt(T)] 
\min_{T \in \cS^{(1)}} \wt(T)
= \int_0^1 {\beta}_0(t) dt.  
\end{equation}
Here, ${\beta}_0(t)$ is the reduced Betti number of $K_n(t)$, i.e., the rank of the reduced homology 
${H}_0(K_n(t))$, and is equal to the number of connected components in $K_n(t)$ minus $1$. 

This formula is \textit{deterministic} in the sense that 
it is valid for any realization of $\{t_e\}_{e \in E_n}$ and hence for 
the induced filtration $\{K_n(t)\}_{t \in [0,1]}$ of $K_n$. 
Given birth times $\{t_e\}_{e \in E_n}$, 
the reduced Betti number ${\beta}_0(t)$ decreases by $1$ at time $t_e$ 
if two connected components in $K_n(t_e-)$ are joined by adding the edge $e$.   
Since connected components can be regarded as generators of the $0$-th
homology, such a time $t_e$ is viewed as a \textit{death time} of 
the corresponding homology generator and the right-hand side  of (\ref{eq:minwgttree}) 
gives the \textit{lifetime sum} of ${H}_0(K_n(t))$. 
%This formula also suggests that studying not just 
%Of particular importance of this formula is not only studying individual
%Erd\"os-R\'enyi graphs for each $t$, but also extending them into
%filtrations. 
This observation naturally leads us to the notion of \textit{persistent
homology}.  

The persistent homology \cite{elz,zc} (see Section \ref{sec:ph} for details) has
recently been studied as a tool to describe how topological features
behave in a filtered topological space. In particular, it provides
the concepts of the {\it birth} and {\it death times} of each topological feature, which measure the appearance and disappearance of the feature in the filtration. The {\it lifetime} is also defined as the difference between the birth and death
times, and it measures the persistence of the feature in the
filtration.

In this paper, we first show the following theorem as a higher dimensional extension of the formula (\ref{eq:minwgttree}), which is not just a counterpart of (\ref{eq:minwgttree}) but also sheds new light on the link among the lifetime sum of the persistent homology, the weights of the minimum spanning acycles, and the integrated Betti numbers.
%In this paper, we first derive a higher dimensional extension of the formula (\ref{eq:minwgttree}). 
%The formula links the lifetime sum of the persistent homology, the weights of the minimum spanning acycles, and the integrated Betti number, and is stated as follows:
%The higher dimensional formula of (\ref{eq:minwgttree}) is
%derived by two representations of the lifetime sum of the persistent
%homology.

%Of particular importance of this formula is not only studying individual
%Erd\"os-R\'enyi graphs for each $t$, but also extending them into
%filtrations or, equivalently, processes. 
%Namely, the right-hand side characterizes the behavior of connected
%components in the filtration, and hence studies Erd\"os-R\'enyi graphs
%extended by the filtration parameter.  

\begin{thm}\label{thm:lifetime2_intro}
Let $X$ be a finite simplicial complex satisfying 
\[
{\beta}_{d-1}(X^{(d)}) 
= {\beta}_{d-2}(X^{(d-1)}) 
= 0
\]
for some $1\leq d\leq \dim X$.
Let $\cX = \{X(t)\}_{t\in\positivereal}$ be a filtration of $X$. 
Then, the following identities hold: 
\begin{align}
L_{d-1} 
&= \min_{T \in \cS^{(d)}} \wt(T) - 
\max_{S \in \cS^{(d-1)}} \wt(X_{d-1} \setminus S) \label{eq:int_intro1}\\
&= \int_0^\infty {\beta}_{d-1}(t) dt. \label{eq:int_intro2}
\end{align}
\end{thm}
%\note[S]{In (1.4), $S$ to $T$ and $L$ to $S$. }
%\note[H]{inconsistency of the integral interval}
Here, $\positivereal$ is the set of nonnegative reals,  
${\beta}_{k}(t)$ is the $k$-th Betti number of $X(t)$, 
$X_{k}$ is the set of $k$-simplices in $X$, 
$X^{(k)}$ is the $k$-dimensional skeleton of $X$, 
$L_{k}$ is the lifetime sum of the $k$-th persistent homology (defined in (\ref{eq:sum})), and $\cS^{(k)}$ is the set of $k$-spanning acycles (Definition~\ref{defn:max}). The proof of this theorem is given in Section \ref{sec:lifetimeformula2}.

The formula (\ref{eq:minwgttree}) is given as a special case $d=1$ of this theorem. It should be remarked that, although only the death times are treated in the Erd\"os-R\'enyi graph process ($d=1$), we also need to study the birth times in the higher dimensional case. This effect causes the second term in (\ref{eq:int_intro1}), and the formulation using the birth and death times in persistent homology fits this extension well. 
Furthermore, we remark that Frieze's theorem can also be expressed by using the lifetime sum as follows:
\[
\E[L_0]\rightarrow \zeta(3)\quad {\rm as}\quad n\rightarrow \infty.
\]

%Based on this deterministic formula, we study the higher dimensional extension 
%of the Erd\"os-R\'enyi graph in the so-called random simplicial complexes. 
%The connectivity and acyclicity of graphs, which
%are commonly studied in random graphs, can be interpreted by using $0$th
%and $1$st homology groups, respectively. Then, it is natural to
%generalize classical results in the Erd\"os-R\'enyi graph into analogues
%expressed by higher dimensional homology groups of suitable random
%simplicial complexes. For example, vanishing and non-vanishing of
%homology groups of random simplicial complexes have recently been
%explored in a number of papers (e.g., see the papers
%\cite{kahle,lm,lp,mw} and references therein).
Based on these formulae (\ref{eq:int_intro1}) and (\ref{eq:int_intro2}), we study a higher dimensional generalization 
of the Erd\"os-R\'enyi graph as random simplicial complexes. 
The connectivity and acyclicity of graphs, which
are commonly studied in random graphs, can be interpreted by using $0$-th
and $1$-st homologies, respectively. Then, it is natural to
generalize classical results in the Erd\"os-R\'enyi graph into analogues
expressed by higher dimensional homology of suitable random
simplicial complexes (e.g., see the papers
\cite{kahle,kahle2,lm,lp,mw} and references therein for recent topics of random simplicial complexes).

%This paper studies an extension of the Erd\"os-R\'enyi graph into both
%directions, i.e., filtrations and simplicial complexes.  We consider two
%processes of random simplicial complexes, Linial-Meshulam process
%\cite{lm} and clique complex process \cite{kahle}, both of which can be
%regarded as natural generalizations of the Erd\"os-R\'enyi graph
%process. Precise definitions of these processes are given in Section
%\ref{sec:scprocess}. Our main results in this paper derive higher
%dimensional analogues of Frieze's $\zeta(3)$-limit theorem for both
%processes (Theorem \ref{thm:2} and Theorem \ref{thm:clique}). 
In this paper, we consider two processes of random simplicial complexes, the Linial-Meshulam process
\cite{lm} and clique complex process \cite{kahle}, both of which can be
regarded as natural generalizations of the Erd\"os-R\'enyi graph
process. Precise definitions of these processes are given in Section
\ref{sec:scprocess}. Our main result shows the following higher dimensional generalization of Frieze's $\zeta(3)$-limit theorem in the Linial-Meshulam process.
\begin{thm}\label{thm:main} 
Let $L_{d-1}$ be the lifetime sum of the $(d-1)$-st persistent homology of 
the $d$-Linial-Meshulam process $(d \ge 1)$ on
 $n$-vertices. Then, 
\begin{equation}\label{eq:main}
 \E[L_{d-1}] = 
%\E[\min_{T \in \cS^{(d)}} \wt(T)] = 
O(n^{d-1})
\end{equation}
as $n \to \infty$. 
\end{thm}
For $d=1$, we already know that the limiting value is $\zeta(3)$ from Frieze's theorem and this agrees with (\ref{eq:main}). 
%The proof of this theorem is given in Section \ref{sec:main}.

This paper is organized as follows. The fundamental concepts of homology and persistent homology are explained in Section \ref{sec:ph_lifetime}. Here, the algebraic formulation using graded modules and the analytic formulation using counting measures are introduced for persistent homology, 
and both formulations are used to derive Theorem \ref{thm:lifetime2_intro} and Theorem \ref{thm:main}. 
%These two formulations are used to derive (\ref{eq:int_intro1}) and (\ref{eq:int_intro2}), respectively. 
In Section \ref{sec:det_formula}, we summarize a determinantal formula of  boundary maps by means of spanning acycles. 
Section \ref{sec:lifetimeformula2} is devoted to proving Theorem \ref{thm:lifetime2_intro}. 
In Section \ref{sec:scprocess}, we explain random persistence diagrams as point processes, and then
introduce the $d$-Linial-Meshulam process and the clique complex process.  The proof of Theorem \ref{thm:main} is presented in Section \ref{sec:main}. Furthermore, we also show a partial result (Theorem \ref{thm:clique}) on the higher dimensional extension of Frieze's theorem for the clique complex process. In Section \ref{sec:conclusion}, we list some conjectures and open questions.

\section{Persistent Homology and Lifetime}\label{sec:ph_lifetime}
\subsection{Homology}\label{sec:homology}
We first recall some fundamental concepts of simplicial homology. For more details, the reader may refer to \cite{munkres}.
Let $X$ be a simplicial complex on a finite set $V=\{1,\dots,n\}$, i.e.,
a collection of nonempty subsets of $V$ which includes all elements in
$V$ and is closed under the operation of taking nonempty subsets.  
An element $\sigma \in X$ with $|\sigma| = k+1$ is called a $k$-simplex and
$k$ is called its dimension. The dimension $\dim X$ of the simplicial
complex $X$ is given by the maximum dimension of simplices in $X$.  We
denote the set of $k$-simplices in $X$ and its cardinality by $X_k$ and
$f_k(X)=|X_k|$, respectively. The $k$-dimensional skeleton of $X$ is defined by
$X^{(k)} = \bigsqcup_{j=0}^k X_j$. In this paper, we only deal with finite simplicial complexes, i.e., $|V|<\infty$.

For a simplicial complex $X$, the boundary map $\partial_k: C_k(X)\rightarrow C_{k-1}(X)$ and 
the chain complex 
\begin{equation}\label{eq:chaincomplex}
	\cdots\longrightarrow C_{k+1}(X)\stackrel{\partial_{k+1}}{\longrightarrow} C_k(X) \stackrel{\partial_{k}}{\longrightarrow}C_{k-1}(X)\longrightarrow\cdots
\end{equation}
in the integer coefficient are defined in a standard way.  
For $\sigma=\{v_0,\dots,v_k\}\in X$, we set its oriented simplex by the 
ordering $v_0<\dots<v_k$ and denote it by $\langle\sigma\rangle=\langle v_0\cdots v_k\rangle$.
Then, the $k$-th homology $H_k(X)=Z_k(X)/B_k(X)$ is defined as the
quotient $\Z$-module of $Z_k(X)=\kernel \partial_k$ and
$B_k(X)=\image \partial_{k+1}$.  

In this paper, we use the reduced homology $\tilde{H}_0(X)$ for $k=0$, which is given by $H_0(X)\simeq \tilde{H}_0(X)\oplus\Z$.  For simplicity, we use the same symbol $H_0(X)$ for the $0$-th reduced homology and omit to specify ``reduced" from now on. 
We also note that the homology can be represented as $H_k(X)\simeq T_k(X)\oplus \Z^{\beta_k(X)}$, 
where $T_k(X)$ and ${\beta}_k(X)$ are called the $k$-th torsion and the $k$-th Betti number, respectively.

For simplicial complexes $Y\subset X$, let $C_k(X,Y)=C_k(X)/C_k(Y)$ be the quotient module. The boundary map in the chain complex (\ref{eq:chaincomplex}) naturally induces the relative chain complex
\begin{equation}\label{eq:relativechain}
	\cdots\longrightarrow C_{k+1}(X,Y)\stackrel{\partial_{k+1}}{\longrightarrow} C_k(X,Y) \stackrel{\partial_{k}}{\longrightarrow}C_{k-1}(X,Y)\longrightarrow\cdots.
\end{equation}
Then, the $k$-th relative homology $H_k(X,Y)$ is defined by the same way as $H_k(X,Y)=Z_k(X,Y)/B_k(X,Y)$, where $Z_k(X,Y)=\kernel \partial_k$ and $B_k(X,Y)=\image \partial_{k+1}$ in (\ref{eq:relativechain}).
It is well known that there exists an exact sequence for a pair $Y\subset X$:
\begin{equation}\label{eq:pairlongexact}
	\cdots \longrightarrow H_{k+1}(X,Y)\longrightarrow H_k(Y)\longrightarrow H_k(X)\longrightarrow H_k(X,Y)
	\longrightarrow H_{k-1}(Y)\longrightarrow\cdots.
\end{equation}

\subsection{Persistent Homology}\label{sec:ph}
Let $\Z_{\geq 0}$ and $\positivereal$ be the sets of nonnegative integers and reals, respectively.  
Let $\cX=\{X(t)\mid t\in \positivereal\}$ be a 
right continuous filtration of a simplicial complex $X$. 
Namely, $X(t)$ is a
subcomplex of $X$, $X(t)\subset X(t')$ for $t\leq t'$, and 
$X(t)=\bigcap_{t<t'}X(t')$.
We assume that there exists a saturation time $T$ such that $X(T)=X$.
For each simplex $\sigma\in X$, let
$t_{\sigma}=\min\{t\in\positivereal\mid \sigma\in X(t)\}$ denote the
birth time of $\sigma$.

Let $K$ be a field of characteristic zero, and let $K[\positivereal]$ be a monoid ring. That is, $K[\positivereal]$ is a $K$-vector space of formal linear combinations of elements in $\positivereal$ equipped with a ring structure
\[
	(at)\cdot (bs)=(ab)(t+s),\quad a,b\in K,~t,s\in\positivereal.
\]
In the following, the elements in $K[\positivereal]$ are expressed by linear combinations of (formal) monomials $az^t$, where $a\in K$, $t\in \positivereal$, and $z$ is an indeterminate. Then, the product of two elements are given by the linear extension of $az^{t}\cdot bz^{s}=abz^{t+s}$. %When $\T=\positiveinteger$, $K[\positiveinteger]$ is nothing but the polynomial ring $K[z]$.

%It should be remarked that a jump
%\[
%	\bigcup_{s<t_i} X(s) \subsetneq X(t_i)%_{t_i}
%\]
%of the filtration occurs at a finite subset $\{t_1,\dots,t_m\}$ of times, because of $|X|<\infty$. Hence, by appropriately changing the time $t$, we can use the standard polynomial ring $K[z]=K[\Z_{\geq 0}]$ in the following argument (see \cite{zc}).

For $t\in \positivereal$, let $C_k(X(t))$ be the $K$-vector space spanned by the oriented $k$-simplices in $X(t)$. 
The $k$-th chain group $C_k(\cX)$ of $\cX$ is defined as a graded module over the monoid ring $K[\positivereal]$ by taking a direct sum
\[
	C_k(\cX)=\bigoplus_{t\in \positivereal}C_k (X(t))=
	\{(c_t)\mid c_t\in C_k(X(t)),~t\in\positivereal\},
\]
where the action of a monomial $z^s$ on $C_k(\cX)$ is given by the right shift operator
\[
	z^s\cdot (c_t)=(c'_t),\quad
	c'_t=\left\{\begin{array}{cl}
	c_{t-s},&t\geq s\\
	0,& t<s
	\end{array}\right..
\]

For an oriented simplex $\langle\sigma\rangle$, let us define 
\[
	\llangle\sigma\rrangle = (c_t),\quad
	c_t=\left\{\begin{array}{cl}
	\langle \sigma \rangle,& t=t_\sigma\\
	0,&t\neq t_\sigma
	\end{array}\right..
\]
%Let  $\iota_t: C_\ell(X(t))\rightarrow C_\ell(\cX)$ be an inclusion defined by
%\[
%	\iota_t(\sigma)=(c_r),~~~
%	c_r=\left\{\begin{array}{cl}
%	\sigma,&r=t,\\
%	0,&r\neq t.
%	\end{array}
%	\right.
%\]
Then, the set 
$
	\Xi_k=\{\llangle\sigma\rrangle \mid \sigma\in X_k\}
$
forms a basis of $C_k(\cX)$. 
The boundary map $\delta_k: C_k(\cX)\rightarrow C_{k-1}(\cX)$ is defined by the linear extension of 
%\begin{equation}\label{eq:boundarymap}
%	\delta_\ell(e_\sigma)=\sum_{j=0}^\ell(-1)^jz^{t_\sigma-t_{\sigma_j}}e_{\sigma_j},
%\end{equation}
\begin{equation}\label{eq:boundarymap}
	\delta_k\llangle \sigma\rrangle=\sum_{j=0}^k(-1)^jz^{t_{\sigma}-t_{\sigma_j}}\llangle \sigma_j\rrangle,
\end{equation}
where $\langle\sigma\rangle=\langle v_0\cdots v_k\rangle$ and $\sigma_j=\sigma\setminus \{v_j\}$. We note $t_{\sigma}-t_{\sigma_j}\geq 0$ from $\sigma_j\subset \sigma$. The matrix form of $\delta_k$ using the standard bases $\Xi_k$ and $\Xi_{k-1}$
consists of entries $\pm z^t\in K[\positivereal]$.

The cycle group $Z_k(\cX)$ and the boundary group $B_k(\cX)$ in $C_k(\cX)$ are defined by
\[
	Z_k(\cX)=\kernel \delta_k,\quad B_k(\cX)=\image\delta_{k+1}.
\]
It follows from $\delta_k\circ\delta_{k+1}=0$ that $B_k(\cX)\subset Z_k(\cX)$.
Then, the $k$-th persistent homology is defined by
\[
	H_k(\cX)=Z_k(\cX)/B_k(\cX).
\]
We note that the persistent homology is a graded module over $K[\positivereal]$.

The following theorem is known as the structure theorem of the persistent homology. 
\begin{thm}[\cite{zc}]
There uniquely exist indices $p,q\in\Z_{\geq 0}$ and 
$(b_i,d_i)\in\positivereal^2$ for $i=1,\dots,p$ with $b_i<d_i$ and $b_i\in\positivereal$ for $i=p+1,\dots,p+q$ such that the following isomorphism holds:
\begin{equation}\label{eq:decomposition}
	H_k(\cX)\simeq\bigoplus_{i=1}^p
	\left((z^{b_i})\biggl/(z^{d_i})
	\right)\oplus\bigoplus_{i=p+1}^{p+q}(z^{b_i}),
\end{equation}
where $(z^a)$ expresses an ideal in $K[\positivereal]$ generated by the monomial $z^a$. 
When $p$ or $q$ is zero, the corresponding direct sum is ignored.
\end{thm}

Here $b_i$ and $d_i$ are called the birth and the death times,
respectively, and they measure the events of appearance and
disappearance of topological features in the filtration $\cX$.  
Namely, it expresses that a homology generator is born in $H_k(X(b_i))$, persists in $H_k(X(t))$ for $b_i\leq t\leq d_i$, and dies in $H_k(X(d_i))$. 
The lifetime $l_i$ of the pair $(b_i,d_i)$ is defined by $l_i=d_i-b_i$.
For $p+1\leq i \leq p+ q$, we assign the death time
$d_i=\infty$ as the element of the extended nonnegative reals 
$\overline\R_{\ge 0} := \positivereal \cup \{\infty\}$.  
We remark that the representation (\ref{eq:decomposition}) of the persistent homology is the counterpart to the one using the torsion and free modules in the standard homology. Both are derived from the structure theorem of finitely generated modules over PID. 

The indecomposable decomposition (\ref{eq:decomposition}) of the 
persistent homology can be expressed by using a multiset called the
$k$-th  persistence diagram
\begin{equation}\label{eq:pd}
D_k(\cX)=\{(b_i,d_i)\in\overline{\R}_{\ge 0}^2\mid i=1,\dots,p+q\}.
\end{equation}
Similar to homology, we use the reduced persistent homology for $k=0$, which is defined by deleting one generator with infinite death time from $H_0(\cX)$. For simplicity, we use the same symbol $H_0(\cX)$ and omit to specify ``reduced". 
The persistence diagram $D_0(\cX)$ is also defined in a reduced sense. 

%
%it is convenient to consider the reduced persistent homology $\tilde{H}_0(\cX)$ in a similar manner to
%$\tilde{H}_0(X)$. In this case, $\tilde{H}_0(\cX)$ is determined by deleting one generator with infinite death time from $H_0(\cX)$.  

%When $p + 1 \le i \le p+q$, $d_i$ is set to be the saturation time $T$. 
%In this paper, we denote the sum of the lifetimes $\ell_i=d_i-b_i$ by 
%\begin{equation}
%L_\ell=\sum_{i=1}^{p+q} \ell_i. 
%\end{equation}

\subsection{Lifetime Formula I}\label{sec:lifetime1}
We denote the lifetime sum of the $k$-th persistent homology by 
\begin{equation}\label{eq:sum} 
L_k = \sum_{i=1}^{p+q} (d_i-b_i), 
\end{equation}
where $L_{k}$ is understood as $\infty$ when $q \ge 1$. 
\iffalse
The persistent homology treated in this paper does not have the
latter part in the indecomposable decomposition
(\ref{eq:decomposition}). 
Hence, we often suppose the case $q=0$. 
\fi
%We denote the sum of the lifetimes by 
%\begin{equation}\label{eq:sum} 
%L_\ell=\sum_{i=1}^p(d_i-b_i).
%\end{equation}

It is often convenient to regard the $k$-th persistence diagram (\ref{eq:pd}) as a counting measure 
\[
	\xi_k = \sum_{0 \le x < y \le \infty} m_{(x,y)} \delta_{(x,y)}
\]
on the set $ \Delta = \{(x,y) \in \extendedpositivereal^2 \mid x \le y \}$, where $\delta_{(x,y)}$ is the delta measure at $(x,y)$ and 
\[
 m_{(x,y)} = |\{1 \le i \le p+q \mid (b_i, d_i) = (x,y)\}|
\]
is the multiplicity. 
We note that
\begin{equation}
{\beta}_k(t) = \xi_k([0,t] \times [t,\infty]),
\label{eq:rankofHd}
\end{equation}
where $\beta_k(t)=\beta_k(X(t))$.
%where $\tilde{\beta}_k(t)=\rank \tilde{H}_k(X(t))$.
%We also remark that the point mass $\delta_{(b, \infty)}$ corresponds to the latter direct summand ($q\geq 1$) 
%in (\ref{eq:decomposition}).
%and $\delta_{(b, b)}$ does not appear from persistence diagrams. 

%By definition, $\xi_k$ is supported on the set $\Delta$. 
We write 
\[
 \langle \xi_k, f \rangle = \int_{\Delta} f(x,y) \xi_k(dxdy)
\]
for any measurable function $f : \Delta \to \r$ as long as the right-hand 
side makes sense. 
For example, when $f$ is the indicator function $I_A$ of a measurable
set $A \subset \Delta$, $\langle \xi_k, I_A \rangle = \xi_k(A)$ is the 
number of points inside $A$ counted with multiplicity. 
By setting $f(x,y) = y-x$, we also have 
\begin{equation}
\langle \xi_k, f \rangle = \int_{\Delta} (y-x) \xi_k(dxdy)
= \sum_{i=1}^{p+q} (d_i-b_i) = L_k. 
\label{eq:lifetime0}
\end{equation}
Then, we easily obtain the following formula of the lifetime sum.
 
\begin{prop}\label{prop:lifetime1}
%Let $\xi_k$ be the $k$-th reduced persistence diagram of a filtration $\cX = \{X(t)\}_{t\in \positivereal}$. Then, 

\[
%L_d = \int_0^{\infty} \rank \tilde{H}_d(X(t)) dt. 
L_k = \int_{[0,\infty]} {\beta}_k(t) dt. 
\] 
\end{prop}
%$m_{x,y}$ is the multiplicity of generators with 
%birth time $x$ and death time $y$. It is known that (is it right?) 
%It is easy to see that 
%\[
%\rank \tilde{H}_{\ell-1}(X^{(\ell)}(t)) = \xi([0,t] \times [t,\infty)), 
%\]
%where $\cX = (X^{(\ell)}(t))_{t \ge 0}$. 
\begin{proof}
By Fubini's theorem, 
(\ref{eq:rankofHd}) and (\ref{eq:lifetime0}), we see that 
%\begin{align*}
%\lefteqn{\int_{[0,\infty]} \rank \tilde{H}_k(X(t)) dt} \\ 
%&= \int_{[0,\infty]} dt \int_{\positivereal^2} I_{[0,t]}(x) I_{[t,\infty]}(y) \xi_k(dxdy) \\
%&= \int_{\positivereal^2} \xi_k(dxdy) \int_{[0,\infty]} I(0 \le x \le t \le y \le \infty) dt \\
%&= \int_{\positivereal^2} (y-x) \xi_k(dxdy) \\
%&= L_k. 
%\end{align*}
\begin{eqnarray*}
L_k&=&\int_{\Delta} (y-x) \xi_k(dxdy) \\
&=&\int_{\Delta} \xi_k(dxdy) \int_{[0,\infty]} I(0 \le x \le t \le y \le \infty) dt \\
&=&\int_{[0,\infty]} dt \int_{\Delta} I_{[0,t]}(x) I_{[t,\infty]}(y) \xi_k(dxdy) \\
&=&\int_{[0,\infty]} \beta_k(t) dt. 
\end{eqnarray*}
When $q \ge 1$, the both sides are $\infty$. 
\end{proof}

The persistent homology treated in this paper does not have the
latter part in the indecomposable decomposition (\ref{eq:decomposition}). 
Hence, we always suppose the case $q=0$ from now on. 

\begin{rem}{\rm 
This lifetime formula can be regarded as Little's formula in queuing theory \cite{little}.
}
\end{rem}

\begin{rem}\label{rem:landscape}
{\rm 
The lifetime sum $L_k$ 
can be regarded as the $\ell^1$-norm $\|\vec{l}\|_1$ of a sequence $\vec{l} = (l_i)_{i=1}^p$ of 
lifetimes $l_i=d_i-b_i$ in the $k$-th  persistent homology. 
P.~Bubenik points out that 
the squared $\ell^2$-norm of $\vec{l}$ 
is equal to $4$ times the $L^1$-norm of the persistent landscape \cite{B}.
In order to make clear the connection to the persistent landscape, we derive 
a similar integral formula for the $\ell^2$-norm. 

First, let us define the $(t-s)$-persistent homology \cite{elz} 
\[
	H_k(s,t) = Z_k(X(s)) / (B_k(X(t)) \cap Z_k(X(s))).
\]
%and its reduced version $\tilde{H}_k(s,t)$. 
We note that 
\[
\rank {H}_k(s,t)=\int_{\Delta} 
I_{[0, s]}(x) I_{[t,\infty]}(y)  \xi_k(dxdy),
\]
and we denote the left-hand side by $\beta_k(s,t)$. 
Then, the integral formula for the $\ell^2$-norm is given by
\[
\|\vec{l}\|_2^2 
= 2 \int_{0 \le s \le t \le \infty} {\beta}_k(s,t) ds dt.
\]
This formula is derived in a similar way using Fubini's theorem:
\begin{align*}
\|\vec{l}\|_2^2 
&= \int_{\Delta} (y-x)^2 \xi_k(dxdy) \\
&= \int_{\Delta} \xi_k(dxdy) \left(\int_{[0,\infty]} I(0 \le x \le t \le y \le \infty) dt\right)^2 \\
%&= \int_{\T^2} \xi_k(dxdy) \int_{[0,\infty]} I(0 \le x \le t \le y \le \infty) dt 
%\int_{[0,\infty]} I(0 \le x \le s \le y \le \infty) ds \\
&= \int_{[0,\infty]^2} dt ds \int_{\Delta} 
I_{[0,t]}(x) I_{[t,\infty]}(y) I_{[0,s]}(x) I_{[s,\infty]}(y) \xi_k(dxdy) \\ 
&= \int_{[0,\infty]^2} dt ds \int_{\Delta}
I_{[0,t \wedge s]}(x) I_{[t \vee s,\infty]}(y)  \xi_k(dxdy) \\ 
&= 2 \int_{0 \le s \le t \le \infty} \int_{\Delta}
I_{[0, s]}(x) I_{[t,\infty]}(y)  \xi_k(dxdy) \\
 &= 2 \int_{0 \le s \le t \le \infty} \beta_k(s,t) ds dt,
\end{align*}
where $a \wedge b = \min(a,b)$ and $a \vee b = \max(a,b)$. 
}
\end{rem}

\section{Spanning Acycle and Determinantal Formula}\label{sec:det_formula}
In this section, we basically follow the argument in \cite{matrix_tree_thm}.

\subsection{Spanning Acycle} \label{sec:spanningacycle}
%Throughout this section, $X$ is assumed to be a $d$-dimensional simplicial complex. 
%We denote the set of $k$-dimensional faces of $X$ by $X_k$ ($
%k=0,1,\dots, d$) and the $k$-dimensional skelton by $X^{(k)}$, 
%i.e., $X^{(k)} = \sqcup_{j=0}^k X_j$.  
Let $X$ be a simplicial complex and let $k\in\N$ be ${k\leq \dim X}$.
For a subset $S \subset X_k$, 
we define a $k$-dimensional subcomplex of $X$ by 
\begin{equation}
X_S = S \sqcup X^{(k-1)}. 
\label{eq:defofXT}
\end{equation}
%In this paper, we adopt the following definition of $k$-spanning acycle. 

\begin{defn}\label{defn:max}
{\rm 
A subset $S\subset X_k$ is called a $k$-spanning acycle if 
\begin{enumerate}[(a)]
\item ${H}_k(X_S)=0$, and
\item $|{H}_{k-1}(X_S)| < \infty$. 
\end{enumerate}
%where 
%\begin{equation}
%\gamma_k(X) 
%= f_k(X) - \tilde{\beta}_k(X^{(k)}) +
% \tilde{\beta}_{k-1}(X^{(k)})
%\label{eq:gamma} 
%$\end{equation}
%with $\beta_{-1}(X^{(0)})=0$. 
The set of $k$-spanning acycles in $X$ is denoted by $\cS^{(k)}$. 
%We write $\cS_c^{(k)} = \{X_k \setminus S \ | \ S \in \cS^{(k)}\}$. 
}
\end{defn}

This definition is a natural generalization of the spanning trees of a graph. 
For $\dim X=1$ and $k=1$, $S$ is a subset of edges
and $X_S = V \sqcup S$ is a graph. 
In this case, the conditions (a) and (b) are equivalent that $X_S$ has
 no cycles and $X_S$ is connected, respectively. 
 This means that the $1$-spanning acycle $S$ is nothing but a spanning tree.

\begin{rem}{\rm 
This definition is originally introduced by Kalai \cite{kalai} for $X$ being a $k$-dimensional simplicial complex with the  complete $(k-1)$-skeleton. 
This is  essentially the same as \textit{simplicial spanning tree} given 
in \cite{matrix_tree_thm}. 
%In \cite{lyons}, instead of $k$-spanning acycles, the notion of \textit{$k$-bases} 
%is considered in the context of matroids.
%% as $S$ with $\ker (\partial_k)|_{S} = 0$, which always exists. 
%These notions may coincide in many cases. 
}\end{rem}

\begin{ex}{\rm 
Let $\sigma$ be a $3$-simplex and let $X$ be the simplicial complex
 consisting of all proper subsets in $\sigma$. Then, any collections of
 three $2$-simplices in $X$ become $2$-spanning acycles. 
 On the other hand, any collections of two $2$-simplices are not $2$-spanning
 acycles. 
More generally, the set of the $2$-simplices in a $2$-dimensional triangulated sphere minus one $2$-simplex forms a $2$-spanning acycle. }
\end{ex}

\begin{lem}\label{lem:necessary1}
If there exists a $k$-spanning acycle $S$ in $X$, 
then ${|{H}_{k-1}(X^{(k)})| < \infty}$. 
%then $\tilde{\beta}_{k-1}(X^{(k)})=0$. 
\end{lem}
\begin{proof}
It follows from (\ref{eq:defofXT}) that $C_k(X_S) \subset C_k(X)$ and $C_j(X_S) = C_j(X)$
for $j<k$. Hence, we have 
$\image \partial_k|_S \subset \image \partial_k$ and 
$\kernel \partial_{k-1}|_S = \kernel \partial_{k-1}$, 
where $\partial_{k}|_S$ expresses the restriction of $\partial_{k}$ on $C_{k}(X_S)$. 
This implies that there exists a surjection from 
${H}_{k-1}(X_S)$ to ${H}_{k-1}(X^{(k)})$. 
Hence, if $S$ is a $k$-spanning acycle, the condition (b) implies 
$|{H}_{k-1}(X^{(k)})| < \infty$. 
\end{proof}

\begin{rem}{\rm 
%From Lemma~\ref{lem:necessary1}, there are no $2$-spanning acycles for $2$-dimensional simplicial complexes 
%via triangulations of orientable surfaces of genus $g \ge 1$. 
%In \cite{lyons}, instead of $k$-spanning acycles, the notion of $k$-bases is considered in the context of matroids.  In that definition, for example,  the set of the 2-simplices in a triangulation of a 2-dimensional oriented surface with genus $g \ge 1$ minus one 2-simplex forms a 2-base. 
In \cite{lyons}, instead of $k$-spanning acycles, the notion of $k$-bases is considered in the context of matroids.  In that definition, for example,  the set of the 2-simplices in a 2-dimensional triangulated oriented surface with genus $g \ge 1$ minus one 2-simplex forms a 2-base, whereas  there are no $2$-spanning acycles in our definition from Lemma~\ref{lem:necessary1}.
}\end{rem}

For $k\ge 0$, let us define
\[
\gamma_k(X)= f_k(X^{(k)}) - {\beta}_k(X^{(k)}) + {\beta}_{k-1}(X^{(k)})
\]
with ${\beta_{-1}(X^{(0)})=0}$. 
Then, we obtain a complementary characterization of $k$-spanning acycles as follows.

\begin{lem}\label{complementary}
Any two of the three conditions {\rm (a)}, {\rm (b)} in Definition~\ref{defn:max} 
and 
\begin{align}
|S|=\gamma_k(X)
\label{eq:gamma} 
\end{align}
imply the third. 
\end{lem}
\begin{proof}
First, we note from (\ref{eq:defofXT}) that
\begin{align*}
 f_j(X_S) &= f_j(X^{(k)}), \quad\quad 0\le j \le k-1,\\
{\beta}_j(X_S) &= {\beta}_j(X^{(k)}), \quad\quad 0\le j \le k-2. 
\end{align*}
By the Euler-Poincar\'e formula, we see that 
\begin{align*}
 \chi(X_S) - \chi(X^{(k)}) 
&= (-1)^k \{f_k(X_S) - f_k(X^{(k)})\} \\
&= \sum_{j=k-1}^k (-1)^j \{{\beta}_j(X_S) - {\beta}_j(X^{(k)})\}, 
\end{align*}
where $\chi$ is the Euler characteristic. 
This is equivalent to 
\[
 \{|S| - \gamma_k(X)\} + 
{\beta}_{k-1}(X_S) - {\beta}_k(X_S)=0, 
\]
and the assertion is obvious from this identity. 
\end{proof}

The cardinality of a $k$-spanning acycle is given as follows. 
\begin{cor}
 If $S$ is a $k$-spanning acycle, then 
\begin{equation}
|S| = \gamma_k(X) = f_k(X^{(k)}) - {\beta}_{k}(X^{(k)}).  
\label{eq:kacyclenum}
\end{equation}
\end{cor}
\begin{proof}
The claim follows from Lemma~\ref{lem:necessary1} and
 Lemma~\ref{complementary}. 
\end{proof}

\begin{lem} For $k \ge 0$, 
\[
 \gamma_k(X) = \dim \kernel \partial_{k-1} 
- \delta_{k,1} + \delta_{k,0}, 
\]
and for $k \ge 1$, 
\begin{equation}
f_{k-1}(X) - \gamma_k(X)
= \gamma_{k-1}(X) - {\beta}_{k-2}(X^{(k-1)}).  
\label{eq:fgamma}
 \end{equation}
\end{lem}
\begin{proof}
Set $N_k = \dim \kernel \partial_k$ and 
$I_k = \dim \image \partial_k$. 
Then, we have
\begin{align*}
\gamma_k(X)
&= (N_k + I_k) - (N_k - \delta_{k,0})+ (N_{k-1} - \delta_{k,1} - I_k)\\
&= N_{k-1} - \delta_{k,1} + \delta_{k,0}. 
\end{align*}
Similarly, for $k \ge 1$, 
\begin{eqnarray*}
\lefteqn{f_{k-1}(X) - \gamma_k(X) - \gamma_{k-1}(X) + {\beta}_{k-2}(X^{(k-1)})} \\
&=& (N_{k-1} + I_{k-1}) - (N_{k-1} - \delta_{k,1} + \delta_{k,0}) \\ 
&& - (N_{k-2} - \delta_{k-1,1} + \delta_{k-1,0})
+ (N_{k-2} - I_{k-1} - \delta_{k-2,0}) \\
&=& - \delta_{k,0} = 0.
\end{eqnarray*}
\end{proof}

\begin{prop}
 Let $X$ be a simplicial complex satisfying 
\begin{equation}
 {\beta}_{j-1}(X^{(j)}) = 0, \quad 1 \le j < \dim X. 
\label{eq:acyclic}
\end{equation}
Then, for $0 \le k \le \dim X$, 
\begin{equation}
 \gamma_k(X) = (-1)^k \left\{1 - \sum_{j=0}^{k-1} (-1)^j f_j(X)\right\}. 
\label{eq:gamma_k} 
\end{equation}
\end{prop}
\begin{proof}
The equality obviously holds for $k=0$.
%It is obvious that $\gamma_0(X)=1$. 
From (\ref{eq:fgamma}) and (\ref{eq:acyclic}), 
we have 
\[
 f_{j-1}(X) = \gamma_j(X) + \gamma_{j-1}(X)
\]
for $1 \le j \le \dim X$. Taking the alternating sum of 
the above leads to
\[
 \sum_{j=1}^k (-1)^{j-1} f_{j-1}(X) 
= \gamma_{0}(X) - (-1)^k \gamma_k(X), 
\]
and (\ref{eq:gamma_k}) follows from this. 
\end{proof}

\begin{ex}\label{ex:cardinality}
{\rm
Let $X$ be the $(n-1)$-dimensional maximal simplicial complex on $n$ vertices. 
It is obvious that ${H}_{k-1}(X^{(k)}) = 0$ 
(and thus ${\beta}_{k-1}(X^{(k)})=0$) for
 $k=1,2,\dots,n-1$ and $f_k(X) = {n \choose k+1}$. 
Then, it follows from (\ref{eq:gamma_k}) that 
\[
       \gamma_k(X) = {n-1 \choose k} 
      \]
for $k=0,1,\dots,n-1$. }
\end{ex}

\subsection{Determinantal Formula}\label{sec:determinantalformula}
%For $K\subset X_{d-1}$ and $S\subset X_d$, we denote by $\partial_{KS}$ the submatrix of the boundary matrix $\partial_d$ restricted to the columns and rows spanned by the simplices of $K$ and $S$.
%Let $X$ be a simplicial complex.  
Let $d\in \N$ be $d\leq \dim X$. Let us express the boundary map $\partial_d: C_d(X)\rightarrow C_{d-1}(X)$ in the matrix form under the standard bases (the sets of oriented simplices). 
For $K\subset X_{d-1}$ and $S\subset X_d$, we denote by $\partial_{KS}$ the submatrix of $\partial_d$ restricted to the rows and columns spanned by the simplices in $K$ and $S$, respectively. 
The submatrices $\partial_K$ and $\partial_S$ are similarly defined. 
\begin{lem}\label{lem:dks}
Let $K, L \subset X_{d-1}$ with $K = X_{d-1} \setminus L$ and $S \subset
 X_d$. Suppose that $|K|=|S| = \gamma_d(X)$. 
Then, $\det \partial_{KS}\neq 0$ if and only if $S\in \cS^{(d)}$ and
 ${H}_{d-1}(X_L)=0$. In this case, 
\[
 |\det \partial_{KS}| = |H_{d-1}(X_S, X_L)|. 
\]
\end{lem}
\begin{proof}
It follows from (\ref{eq:defofXT}) that $X_L$ is a subcomplex of $X_S$ and
\[
% (X_S)^{(k)} = (X_L)^{(k)} = X^{(k)},\quad 0\le  k \le d-2. 
 (X_S)_{k} = (X_L)_{k} = X_{k},\quad 0\le  k \le d-2. 
\]
This implies $H_{k}(X_S,X_L) = 0$ for $0\le k \le d-2$. 
Then, we have an exact sequence (see (\ref{eq:pairlongexact}))
%\begin{align}
%0 
%&\rightarrow \underbrace{\tilde{H}_{d}(X_L)}_{=0} 
%\rightarrow \tilde{H}_{d}(X_S)\rightarrow 
%H_{d}(X_S,X_L) \nonumber \\
%&\rightarrow \tilde{H}_{d-1}(X_L)\rightarrow \tilde{H}_{d-1}(X_S)\rightarrow 
%H_{d-1}(X_S,X_L) \label{eq:longexact}\\
%&\rightarrow\tilde{H}_{d-2}(X_L)\rightarrow\tilde{H}_{d-2}(X_S) 
%\rightarrow 0.\nonumber  
%\end{align}
\begin{align}\label{eq:longexact}
0\rightarrow {H}_{d}(X_S)\rightarrow 
H_{d}(X_S,X_L)
\rightarrow {H}_{d-1}(X_L)\rightarrow {H}_{d-1}(X_S)\rightarrow \cdots.
\end{align}

Suppose that ${H}_d(X_S) = {H}_{d-1}(X_L) =0$. Then, the exact sequence (\ref{eq:longexact}) leads to
\[
\kernel \partial_{KS}=H_d(X_S,X_L) =0,
\]
which implies $\det \partial_{KS}\neq 0$.

Assume to the contrary that $\det \partial_{KS}\neq 0$. Then, 
$H_d(X_S,X_L)=\kernel \partial_{KS}=0$, and hence 
${H}_d(X_S) = 0$ from (\ref{eq:longexact}). 
This together with $|S| = \gamma_d(X)$ means $S\in\cS^{(d)}$ from 
Lemma~\ref{complementary}, and ${H}_{d-1}(X_S)$ is a finite group. 
Furthermore, (\ref{eq:longexact}) leads to an injection 
\[
	0\rightarrow {H}_{d-1}(X_L)\rightarrow {H}_{d-1}(X_S).
\]
Because of $\dim X_L=d-1$, ${H}_{d-1}(X_L)$ is free. 
Thus, ${H}_{d-1}(X_L)$ must be zero.

%Recall that $\partial_{KS} : C_d(X_S, X_L) \to C_{d-1}(X_S, X_L)$. 
It follows from $C_{d-2}(X_S, X_L)=0$ that 
\[
H_{d-1}(X_S, X_L) = C_{d-1}(X_S, X_L) / \image \partial_{KS}.
\]
Hence, $\kernel \partial_{KS} = 0$ implies that $H_{d-1}(X_S, X_L)$ is a finite group of order 
$|\det \partial_{KS}|$. 
\end{proof}

\begin{rem}\label{rem:lgamma} {\rm 
Let $K$ and $L$ be as in Lemma~\ref{lem:dks}. 
It follows from (\ref{eq:fgamma}) that
\begin{align*}
 |L| 
&= f_{d-1}(X) - \gamma_d(X) \\
&= \gamma_{d-1}(X) - {\beta}_{d-2}(X^{(d-1)}).
\end{align*}
Hence, we need the condition ${\beta}_{d-2}(X^{(d-1)}) = 0$ 
for $L \in \cS^{(d-1)}$ as stated in Lemma~\ref{lem:necessary1}. 
%In this case, 
%\[
%f_{d-1}(X) = \gamma_d(X) + \gamma_{d-1}(X). 
%\]
%If $X$ is a graph, i.e., $d=1$, then 
%$\tilde{\beta}_{d-2}(X^{(d-1)}) = 0$.  
}
\end{rem}

\iffalse
\begin{rem}\label{rem:lgamma} 
Let $K, L$ be as in Lemma~\ref{lem:dks}. 
\gamma_d(X)$, by (\ref{eq:fgamma}), 
\begin{align*}
 |L| 
&= f_{d-1}(X) - \gamma_d(X) \\
&= \gamma_{d-1}(X) - {\beta}_{d-2}(X^{(d-1)})
\end{align*}
Hence, we need the condition ${\beta}_{d-2}(X^{(d-1)}) = 0$ 
for $L \in \cS^{(d-1)}$. 
If $X$ is a graph, i.e., $d=1$, then 
${\beta}_{d-2}(X^{(d-1)}) = 0$.  
\end{rem}
\fi

\begin{cor}\label{cor:nonzero_det}
Let $X$ be a simplicial complex with 
${\beta}_{d-2}(X^{(d-1)}) = 0$. Suppose 
$K \subset X_{d-1}$ and $S \subset X_d$ satisfy $|K|=|S| =
 \gamma_d(X)$. Then, $\det \partial_{KS}\neq 0$ if and only if $S\in \cS^{(d)}$ and $X_{d-1} \setminus K \in \cS^{(d-1)}$. 
%Let $K \subset X_{d-1}$ and $S \subset X_d$ satisfy $|K|=|S| =
% \gamma_d(X)$, and suppose $\tilde{\beta}_{d-2}(X^{(d-1)}) = 0$. 
%%\begin{equation}
%%% \tilde{\beta}_{d-1}(X^{(d)}) = \tilde{\beta}_{d-2}(X^{(d-1)}) = 0. 
%% \tilde{\beta}_{d-2}(X^{(d-1)}) = 0. 
%%\label{eq:tree_exist}
%%\end{equation}
%Then, $\det \partial_{KS}\neq 0$ if and only if $S\in \cS^{(d)}$ and 
%$X_{d-1} \setminus K \in \cS^{(d-1)}$. 
\end{cor}
\begin{proof}
The assertion immediately follows from Lemma~\ref{lem:dks} and 
Remark~\ref{rem:lgamma}. 
\end{proof}

%\begin{ass}\label{ass:finite}
%In what follows, we assume that 
%$X$ is a $d$-dimensional simplicial complex with 
%the condition (\ref{eq:tree_exist}). 
%\end{ass}

Let $\x = (x_\sigma)_{\sigma \in X_{d-1}}$ and $\y = (y_\eta)_{\eta \in
X_d}$ be indeterminates corresponding to the $(d-1)$-simplices and
the $d$-simplices in $X$, respectively. Set 
\begin{equation}
 \partial_d(\x,\y) = \diag(\x) \ \partial_d \ \diag(\y),  
\label{eq:delxy}
\end{equation}
where $\diag(\x)$ is the diagonal matrix with entries being $\x$. 
%We write $\partial_d$ for $\partial_d(\mathbf{1}, \mathbf{1})$. 

\begin{prop}\label{prop:det_expansion}
%Let $X$ be a simplicial complex.
% with ${\beta}_{d-1}(X^{(d)})=0$.
Let $K\subset X_{d-1}$ with $|K|=\gamma_d(X)$.   Then,
\[
\det \partial_d(\x,\y)_K \partial_d(\x,\y)_K^t
= \sum_{S\in\cS^{(d)}}(\det \partial_{KS})^2 \x_K^2 \y_S^2,  
%\\&= \sum_{S\in\cS^{(d)}} |H_{d-1}(X_S, X_L)|^2 X_K Y_S,
 \]
where $\x_K = \prod_{\sigma \in K} x_{\sigma}$ 
and $\y_S = \prod_{\eta \in S} y_{\eta}$
\end{prop}

\begin{proof}
%We have $f_d(X^{(d)})\ge |K|$ from $|K|=\gamma_d(X)=f_d(X^{(d)})-{\beta}_d(X^{(d)})$. 
The Binet-Cauchy formula leads to
\begin{align*}
\det \partial_d(\x,\y)_K \partial_d(\x,\y)_K^t
%&= \sum_{S \subset X_d, |S|=\gamma_d(X)}(\det \partial_{KS})^2 \x_K^2 \y_S^2
&= \sum_{\substack{S \subset X_d\\|S|=\gamma_d(X)}}(\det \partial_{KS})^2 \x_K^2 \y_S^2
 \\ 
&= \sum_{S \in \cS^{(d)}}(\det \partial_{KS})^2 \x_K^2 \y_S^2.
\end{align*}
The second equality follows from Lemma \ref{lem:dks}. 
\end{proof}

%\begin{rem}
%The both sides are $0$ when $K$ is not the complement of 
%$L \in \cS^{(d-1)}$. 
%\end{rem}

%Let us recall an easy consequence from the homological algebra. 
%\begin{lem}\label{lem:abcd}
%Let 
%\[
%0 \rightarrow 
%A_1 \stackrel{\phi_1}{\rightarrow} 
%A_2 \stackrel{\phi_2}{\rightarrow} 
%A_3 \stackrel{\phi_3}{\rightarrow} A_4
%\rightarrow 0
%\] 
%be an exact sequence. If $A_1, A_3, A_4$ are finite, then $A_2$ is
% finite and 
%\[
% |A_2| = \frac{|A_1| \cdot |A_3|}{|A_4|}. 
%\]
%\end{lem}
%\begin{proof}
%From the first isomorphism theorem, 
%\[
%\frac{|A_i|}{|\kernel \phi_i|} = |\image \phi_i| \quad (i=1,2,3) 
%\]
%and, from exactness, $|\ker \phi_1| = 1$, $|\image \phi_3| = |A_4|$ and 
%\[
%|\image \phi_i | = |\kernel \phi_{i+1}| \quad (i=1,2). 
%\]
%Therefore, we obtain the assertion. 
%\end{proof}

\begin{lem}
Suppose that $S \in \cS^{(d)}$ and 
$L \in \cS^{(d-1)}$, and set $K = X_{d-1}
 \setminus L$. 
%If $|\tilde{H}_{d-2}(X_S)| < \infty$, then 
Then,
\begin{equation}\label{eq:det_homology}
 |\det \partial_{KS}|  
= \frac{|{H}_{d-1}(X_S)| \cdot |{H}_{d-2}(X_L)|}{|{H}_{d-2}(X_S)|}.
\end{equation}
\end{lem}
\begin{proof}
It follows from $L \in \cS^{(d-1)}$ that we have an exact sequence
%\begin{align*}
%0 & \rightarrow \tilde{H}_{d-1}(X_S)\rightarrow 
%\tilde{H}_{d-1}(X_S,X_L) \\
%&\rightarrow\tilde{H}_{d-2}(X_L)\rightarrow\tilde{H}_{d-2}(X_S) 
%\rightarrow 0. 
%\end{align*}
\[
0  \rightarrow {H}_{d-1}(X_S)\rightarrow 
{H}_{d-1}(X_S,X_L) 
\rightarrow{H}_{d-2}(X_L)\rightarrow{H}_{d-2}(X_S) 
\rightarrow 0. 
\]
%Then, it follows from $|\tilde{H}_{d-2}(X_S)|<\infty$ that $\tilde{H}_{d-1}(X_S,X_L)$ is finite and 
Then, ${H}_{d-2}(X_S)$ and ${H}_{d-1}(X_S,X_L)$ are finite.
Therefore, we have 
\[
 |\det \partial_{KS}|  = |{H}_{d-1}(X_S,X_L)|
= \frac{|{H}_{d-1}(X_S)| \cdot |{H}_{d-2}(X_L)|}{|{H}_{d-2}(X_S)|}.
\]
%by Lemma~\ref{lem:abcd}. 
\end{proof}

From this lemma and Proposition \ref{prop:det_expansion}, we have the following theorem.
\begin{thm}\label{thm:det_expansion}
%Let $X$ be a simplicial complex with ${\beta}_{d-1}(X^{(d)})=0$.
Suppose $K\subset X_{d-1}$ 
with $|K|=\gamma_d(X)$ and $L=X_{d-1}\setminus K\in\cS^{(d-1)}$. 
 Then,
\begin{equation}\label{eq:det_expansion}
\det \partial_d(\x,\y)_K \partial_d(\x,\y)_K^t
= \sum_{S\in\cS^{(d)}} 
\left(\frac{|{H}_{d-1}(X_S)| \cdot
 |{H}_{d-2}(X_L)|}{|{H}_{d-2}(X_S)|}\right)^2 \x_K^2 \y_S^2. 
 \end{equation}
%where $X_K = \prod_{\sigma \in K} x_{\sigma}^2$ and $Y_S = \prod_{\eta \in S} y_{\eta}^2$.
\end{thm}

\begin{ex}{\rm 
Let $X$ be a triangulation of a $2$-dimensional sphere. 
For  $X_1 \setminus K \in \cS^{(1)}$, we have
%\begin{align*}
%\det \partial_2(\x,\y)_K \partial_2(\x,\y)_K^t 
%&= \x_K^2 \sum_{S\in\cS^{(2)}} |{H}_{1}(X_S)|^2 \y_S^2\\  
%&= \x_K^2 \sum_{S\in\cS^{(2)}} \y_S^2.  
% \end{align*}
\[
\det \partial_2(\x,\y)_K \partial_2(\x,\y)_K^t 
= \x_K^2 \sum_{S\in\cS^{(2)}} |{H}_{1}(X_S)|^2 \y_S^2
= \x_K^2 \sum_{S\in\cS^{(2)}} \y_S^2.  
\]
}\end{ex}

\begin{ex}{\rm 
Let $X$ be the $(n-1)$-dimensional maximal simplicial complex on $n$ vertices and $L$ be a
set of $(d-1)$-simplices ($d<n$) in $X$ with one fixed vertex. Let us set
$\x=(x_\sigma)$ and $\y=(y_\eta)$ to be $x_\sigma =1$ and $y_\eta=1$ for
all $\sigma\in X_{d-1}$ and $\eta\in X_d$. Then, Theorem
\ref{thm:det_expansion} is reduced to the Kalai's result \cite{kalai}.
Namely, because of ${H}_{d-2}(X_L)=0$ and ${H}_{d-2}(X_S)=0$
in this setting, the equality (\ref{eq:det_expansion}) becomes
\[
n^{{n-2 \choose d}}=\sum_{S\in\cS^{(d)}}|{H}_{d-1}(X_S)|^2. 
\]
Here, the left-hand side is derived by showing that the eigenvalues of 
$\partial_d(\x,\y)_K \partial_d(\x,\y)_K^t$ are given by $1$ and $n$
 with multiplicities ${n-2 \choose d-1}$ and ${n-2 \choose d}$,
 respectively.  
The case $d=1$ is the Cayley's formula counting the number of spanning
 trees.  
}\end{ex}

\iffalse
\begin{ex}
Let $X$ be the $(n-1)$-dimensional simplex on 
$n$ vertices. Then, we have a weighted version of Kalai's result:  
\begin{align*}
\lefteqn{\det \partial_d(\x,\y)_K \partial_d(\x,\y)_K^t}\\  
&= \sum_{S\in\cS^{(d)}} 
|{H}_{d-1}(X_S)|^2 X_K Y_S,  
 \end{align*}
\end{ex}
\fi

\section{Lifetime Formula II}\label{sec:lifetimeformula2}
In this section, we give a proof of Theorem \ref{thm:lifetime2_intro}.
Throughout this section, let us set $d\in \N$ as $d\leq \dim X$. 
Furthermore, we assume that the simplicial complex $X$ satisfies
\[
 {\beta}_{d-1}(X^{(d)}) = {\beta}_{d-2}(X^{(d-1)}) = 0.
\]

Let $\cX=\{X(t)\mid t\in\positivereal\}$ be a filtration of $X$. 
A minimum $d$-spanning acycle of the filtration $\cX$ is defined as a spanning acycle $S\in \cS^{(d)}$ with the minimum weight $\wt(S) = \sum_{\sigma\in S} t_\sigma$ among $\cS^{(d)}$, where $t_\sigma$ is the birth time of the simplex $\sigma$.

We denote  by $M$ the matrix form of the $d$-th boundary map $\delta_d$ of the persistent homology $H_*(\cX)$ under the standard bases $\Xi_d,\Xi_{d-1}$. We also denote its evaluation at $z=1$ by $D = M|_{z=1}$, which is a matrix form of $\partial_d$. 
It should be noted that 
\[
\rank\delta_d=\rank\partial_d=f_d(X)-\kernel \partial_d=f_d(X^{(d)})-{\beta}_d(X^{(d)})=\gamma_d(X).
\]
Let us denote the elementary divisors of $M$ by $d_1=z^{e_1},\dots,d_{r}=z^{e_r}$, where $r=\gamma_d(X)$.

\begin{prop}
Let $K\subset X_{d-1}$ with $|K|=\gamma_d(X)$.  
Then, 
\begin{equation}
\det M_K M_K^t
=z^{2e(K)}\sum_{S\in\cS^{(d)}}(\det D_{KS})^2z^{2\tau(S)},
\label{eq:aks}
\end{equation}
where 
%\begin{eqnarray*}
%	&&\tau(S)=\wt(S)-\min_{S\in\cS^{(d)}}\wt(S),\\
%	&&e(K)=\min_{S\in\cS^{(d)}}\wt(S)-\wt(K).
%\end{eqnarray*}
\[
\tau(S)=\wt(S)-\min_{S\in\cS^{(d)}}\wt(S),\quad
e(K)=\min_{S\in\cS^{(d)}}\wt(S)-\wt(K),
\]
and $e(K)$ is nonnegative.
\end{prop}
\begin{proof}
By setting $\x = (z^{-t_\sigma})_{\sigma \in X_{d-1}}$ and 
$\y = (z^{t_\eta})_{\eta \in X_d}$,  
$\partial_d(\x,\y)$ defined in (\ref{eq:delxy}) coincides with 
$\delta_d : C_d(\cX) \to C_{d-1}(\cX)$. 
By Proposition~\ref{prop:det_expansion}, we obtain 
\begin{align*}
\det M_K M_K^t 
&= \sum_{S\in\cS^{(d)}}(\det M_{KS})^2\\
&= \sum_{S\in\cS^{(d)}}(\det D_{KS})^2 z^{2(\wt(S) - \wt(K))} \\
&= z^{2e(K)}\sum_{S\in\cS^{(d)}}(\det D_{KS})^2z^{2\tau(S)}. 
\end{align*}
The claim $e(K)\geq 0$ follows from the fact $t_\sigma\leq t_\eta$ for $\sigma\subset \eta$.
\end{proof}

%Although it is not obvious from the definition that 
%$e(K)$ is nonnegative, it is the case. Indeed, we can show the following. 
\begin{lem} \label{lem:sum_of_elementary_divisors}
For the elementary divisors $d_1=z^{e_1},\dots,d_{r}=z^{e_r}$ of $M$, 
\[
\min_{K \in \cS_c^{(d-1)}} e(K)=e_1+\dots+e_r,
\]
where $\cS_c^{(d-1)} = \{X_{d-1} \setminus L \mid L \in \cS^{(d-1)}\}$. 
\end{lem}
\begin{proof}

Let us note that the product $d_1\cdots d_r$ is equal to the $r$-th determinant divisor
\[
\Delta_r(M) 
=\gcd\{\det M_{KS}\mid K\subset X_{d-1}, S\subset X_d,|K|=|S|=r\}.
\]
Recall from Corollary \ref{cor:nonzero_det} that 
$\det \partial_{KS}\neq 0$ if and only if $S\in \cS^{(d)}$ and $X_{d-1} \setminus K \in \cS^{(d-1)}$.
Then, the exponent of $\Delta_r(M)$ is equal to the $\min_{K \in \cS_c^{(d-1)}} e(K)$, and hence this leads to the formula.
\end{proof}

%\begin{lem}
%Let $S\in\cS^{(d)}$ be minimum. Then, the columns of $M^{(d)}_S$ are a
%base of $\image M^{(d)}$.
%\end{lem}
%\begin{proof}
%%Set $r={n-1\choose d}$, $s={n-1\choose d+1}$, and $t=r+s$. Let us order
%Set $r= \gamma_d(X)$, $s=f_d(X) - \gamma_d(X) = \tilde{\beta}_d(X)$, and $t=r+s$. 
%Let us order the columns of $M^{(d)}=[a_1\dots a_t]$ by $i<j$ if $\sigma_i\in S$ and
%$\sigma_j\notin S$, where $\sigma_i$ is a $d$-simplex corresponding to
%$i$th column. Since $a_1,\dots,a_{r}$ are linearly independent and
%$r=\rank M^{(d)}$, it suffices to show a linear relation of the form
%\[
%	a_i=\sum_{j=1}^{r}c_jz^{s_j}a_j
%\]
%for each $i=r+1,\dots,s$.
%
%Suppose that we have an irredundant linear relation 
%\[
%	z^{s_i}a_i=\sum_{j=1}^{r}c_j z^{s_j}a_j
%\]
%with $s_i>0$. Without loss of generality, let us set
%$s_{r}=0$. Then,
%$\{\sigma_1,\dots,\sigma_{r-1},\sigma_i\}\in\cS^{(d)}$. 
%However, we have 
%$t_{\sigma_i}<t_{\sigma_{r}}$, 
%which contradicts to the minimality of $S$.
%\end{proof}

%
%\remove[H]{Since $X^{(d-1)}$ is complete, we have $q=0$ in
%(). Furthermore, it follows that $p \leq
%\dim\kernel \partial_{d-1}=\rank\partial_d=r$. In case of $p<r$, we
%add $\ell_i=0$ for $i=p+1,\dots,r$ to the list of lifetimes
%$\ell_1,\dots,\ell_p$.}

Now, let us consider the $(d-1)$-st persistent homology $H_{d-1}(\cX)$ and its lifetimes.
Let $p$ and $q$ be the indices appearing in the indecomposable decomposition (\ref{eq:decomposition}) for $H_{d-1}(\cX)$. Because of  ${\beta}_{d-1}(X^{(d)})=0$, we have $q=0$. Furthermore, it follows that $p \leq
\dim\kernel \delta_{d-1}=\rank \delta_d=r$. In case of $p<r$, we
add $l_i=0$ for $i=p+1,\dots,r$ to the list of the lifetimes
$l_1,\dots,l_p$.

\begin{lem}\label{lem:elementary_divisors_lifefimes}
$\{e_1,\dots,e_r\}$ and $\{l_1,\dots,l_r\}$ coincide as multisets.
\end{lem}
\begin{proof}
Let us express the boundary maps
 $C_d(\cX)\stackrel{\delta_d}{\longrightarrow}C_{d-1}(\cX)\stackrel{\delta_{d-1}}{\longrightarrow}C_{d-2}(\cX)$ 
 in the matrix forms by using the standard bases $\Xi_d, \Xi_{d-1}$, and
 $\Xi_{d-2}$. Then, by performing appropriate base changes, $\delta_d$
 is expressed as a smith normal form 
\[
\delta_d=\left[
\begin{array}{c|c}
A & \vector{0}\\\hline
\vector{0} & \vector{0}
\end{array}
\right],\quad A=\diag(z^{e_1},\dots,z^{e_r}),
\]
where $e_i=t_{\sigma_i}-t_{\tau_i}$ is determined by the birth times of corresponding simplices $\sigma_i\in\Xi_d$ and $\tau_i\in\Xi_{d-1}$. We note that $t_{\sigma_i}, i=1,\dots,r,$ give the death times of the persistent homology $H_{d-1}(\cX)$. 
Furthermore, it follows from $\delta_{d-1}\circ\delta_d=0$ that the first $r$ columns of $\delta_{d-1}$ are now expressed  to be zero vectors. It means that $t_{\tau_i}, i=1,\dots,r,$ are the birth times of $H_{d-1}(\cX)$, and hence $e_i, i=1,\dots,r,$ coincide with the lifetimes of $H_{d-1}(\cX)$.
\end{proof}

\noindent
{\it Proof of Theorem \ref{thm:lifetime2_intro}.}
It follows from Lemma \ref{lem:sum_of_elementary_divisors} and \ref{lem:elementary_divisors_lifefimes} that
\[
L_{d-1}=l_1+\dots+l_r=e_1+\dots+e_r=\min_{K\in\cS^{(d-1)}_c}e(K).
\]
Note that the minimum of $e(K)$ is achieved by a minimum spanning acycle $L=X_{d-1}\setminus K\in S^{(d-1)}$. Thus, by combining with Proposition~\ref{prop:lifetime1}, we obtain Theorem \ref{thm:lifetime2_intro}.

\section{Simplicial Complex Process}\label{sec:scprocess}

\subsection{Random Persistence Diagram as Point Process}
First of all, we briefly recall the notion of point processes or random point fields. 
Let $S$ be a locally compact Polish space (locally compact separable
metrizable space) 
and $Q = Q(S)$ be the set of nonnegative integer valued Radon measures on $S$. 
Here, $\xi$ is called a Radon measure if $\xi$ is locally finite in the
sense that $\xi(K) < \infty$ 
whenever $K \subset S$ is compact.  Each element $\xi \in Q$ can be expressed as 
$\xi = \sum_s m_s \delta_s$, where $\delta_s$ is the delta measure at $s$ 
and $m_s \in \Z_{\geq 0}$ stands for multiplicity. 

Let $B_c(S)$ be the space of bounded measurable functions with compact support on $S$. For $f \in B_c(S)$,  we define a coupling of $\xi \in Q$ and $f \in B_c(S)$ by 
\[
\langle \xi, f \rangle = \int_S f(s) \xi(ds) = \sum_{s} m_s f(s). 
\]
In particular, when $f$ is the indicator function $I_A$ of a measurable set $A$, 
$\langle \xi, I_A \rangle = \xi(A)$ is the number of points in $A$ counting with multiplicity. 
A sequence $\{\xi_n\in Q\}_{n\ge 1}$ is said to converge to $\xi$
\textit{valuely} if $\langle \xi_n, f \rangle$ converges to $\langle \xi, f \rangle$ for 
any bounded continuous functions $f$ with compact support.
The space $Q$ is equipped with the topological $\sigma$-algebra 
$\mathcal{B}(Q)$ with respect to the vague topology. 
%$\mathcal{B}(Q)$ with respect to the \note[H]{vague topology}{?}. 
%The space $Q$ is equipped with the $\sigma$-algebra $\mathcal{Q}$ 
%generated by $\{\xi(A), A \in \mathcal{B}(S)\}$. 
A $Q$-valued random variable on a probability space $(\Omega, \F, \P)$
is called a  
point process or a random point field. 

Given a point process on $S$, 
the expectation $\la(A) := \E \xi(A)$ for 
every Borel set $A$
%\note[S]{No. Here $A \in \mathcal{B}(S)$. $\xi$ and $\lambda$ are 
%measures on $S$. }
defines a measure which may be finite or infinite. 
If it is also a Radon measure, $\la$ is said to be 
the mean measure or the intensity measure. 
In this case, we have 
\[
 \E [\langle \xi, f \rangle] = \int_S f(s) \la(ds)
\]
for $f \in B_c(S)$. 
%Intuitively, we may write $\E \xi = \la$. 
We note that the mean measure does not necessarily belong to $Q$. 
Higher moment measures can also be defined.

Let $\cX = (X(t))_{t \in \positivereal}$ be 
an increasing stochastic process defined on a probability space
$(\Omega, \F, \P)$  
taking values in the set of simplicial complexes, i.e., a random 
filtration of a simplicial complex. 
As in Section \ref{sec:ph}, we assume that there is a finite saturation
time $T = T(\omega)$ such that $X(t) = X(T)$ for $t \ge T$ a.s.

As explained in Section~\ref{sec:ph_lifetime},
every filtration associates persistence diagrams on 
%$\Delta=\{(x,y) \in \positivereal^2  \mid x \le y\}$
$\Delta=\{(x,y) \in \overline{\R}_{\geq 0}^2  \mid x \le y\}$.
%\note[S]{Yes. please change it. It is still Polish. 
%It is homeomorphic to a closed triangle, 
%say $\{(x,y) \in [0,1]^2 : 0 \le x \le y \le 1\}$.}
Namely, a random filtration $\cX$ assigns a sequence of $Q$-valued random variables
\[
\Xi =
	\{\xi_{i}\in Q(\Delta)\mid i\in \Z_{\geq 0}\},
\]
where each $\xi_{i}$ is the $i$-th persistence diagram of $\cX$.  
In this case, the mean measure $\la_{i}$ on $\Delta$ turns out to be 
a Radon measure (indeed a totally finite measure), and we call it the $i$-th mean persistence 
diagram. Hence, we have
\[
 \E[\langle \xi_{i}, f \rangle] = \int_{\Delta} f(x,y)
 \la_{i}(dxdy) 
\]
for $f \in B_c(\Delta)$, and it also makes sense for nonnegative measurable functions.
%\note[S]{for every bounded measurable function $f$ on $\Delta$. 
%It still holds for any nonnegative measurable function. }
%\note[S]{No need to write the following in this paper:  
%At least the above formula can be extended to nonnegative 
%functions by the monotone convergence argument. 
%For any nonnegative function $f$, first we consider $f \wedge n$. 
%In this case, the above is true. Taking limit $n \to \infty$, by the
%monotone convergence theorem, we have the above for $f$ itself.  
%So the question below does not cause a problem. }
In particular, this leads to
\[
 \E[L_i] = \int_{\Delta} (y-x) \lambda_{i}(dxdy)
\]
%\note[H]{$f(x,y)=y-x\notin B_c(\Delta)$}
for the lifetime sum $L_i$ of the $i$-th persistent homology.

We consider two generalizations of the Erd\"os-R\'enyi graph process 
as random filtrations of simplicial complexes
and discuss the expectation of the lifetime sum.

\subsection{Linial-Meshulam Process}
We discuss a stochastic process $\{\K^{(d)}(t)\}_{0 \le t \le 1}$ 
studied in \cite{lm}. 
Let $\Delta_{n-1}$ be the $(n-1)$-dimensional maximal simplicial complex on the set $[n]=\{1,2,\dots,n\}$, and let $\Delta_{n-1}^{(d)}$ be its $d$-dimensional skeleton  ($1\leq d\leq n-1$).
Let $\{t_\sigma \mid \sigma \in (\Delta_{n-1})_d \}$ be i.i.d. random variables
uniformly distributed on $[0,1]$, where $(\Delta_{n-1})_d$ is the set of
all $d$-simplices in $\Delta_{n-1}$. We regard $t_\sigma$ as the birth time of the $d$-simplex $\sigma$. 
Let $\{\K^{(d)}(t)\}_{0 \le t \le 1}$ be an increasing 
stochastic process on simplicial complexes defined by 
\begin{align*}
 \K^{(d)}(0)&= \Delta_{n-1}^{(d-1)} , \\
 \K^{(d)}(t) &= \K^{(d)}(0) \sqcup \{\sigma \in (\Delta_{n-1})_d \mid t_\sigma \le t\}.
\end{align*}
The process starts from the $(d-1)$-dimensional skeleton $\Delta_{n-1}^{(d-1)}$ at
time $0$ and ends up with the $d$-dimensional skeleton $\Delta_{n-1}^{(d)}$ at time $1$, i.e.,
\[
\Delta_{n-1}^{(d-1)} = \K^{(d)}(0) \subset \K^{(d)}(t) 
\subset \K^{(d)}(1) = \Delta_{n-1}^{(d)}. 
\]
We call $\{\K^{(d)}(t)\}_{0 \le t \le 1}$ the $d$-Linial-Meshulam
process. 
In particular, the $1$-Linial-Meshulam process is nothing but 
the Erd\"os-R\'enyi graph process mentioned in Section \ref{sec:intro}. 

\begin{rem}{\rm 
Similar process is studied in \cite{HKP}, in which 
the birth times are i.i.d. exponential random variables with mean $1$ 
instead of uniform random variables. The advantage of their choice of 
random birth times is that the process becomes a continuous-time Markov process. 
}
\end{rem}

Let ${\beta}_k(t)$ denote the $k$-th Betti
number of $\K^{(d)}(t)$ at time $t$. 
Note that $\beta_k(t)=0$ for $k=0,1,\dots,d-2$.
We denote by $f_k(t) = f_k(\K^{(d)}(t))$ the number of $k$-simplices in $\K^{(d)}(t)$. 
Then, by applying the Euler-Poincar\'e formula to 
the $d$-Linial-Meshulam process, 
we have
\begin{equation}
{\beta}_d(t) - {\beta}_{d-1}(t) =  f_d(t) -  {n-1 \choose d}.
\label{lmd} 
\end{equation}
We also note that there exist random times 
$\tau_{d-1}, T_d \in [0,1]$ with $\tau_{d-1}\leq T_d$
such that 
\begin{align*}
&{\beta}_{d-1}(0) = {n-1 \choose d}, \quad
{\beta}_{d-1}(t) =0 \text{ for $t \ge \tau_{d-1}$}, \\
&{\beta}_{d}(0) = 0, \quad 
{\beta}_{d}(t) = {n-1 \choose d+1} \ \text{ for $t \ge T_d$}.
\end{align*}
The Betti numbers $\beta_{d-1}(t)$ and $\beta_d(t)$ are non-increasing and non-decreasing in $t$, respectively.

\subsection{Clique Complex Process}

The clique complex ${\rm Cl}(G)$ associated with a graph $G$ 
is the maximal simplicial complex having $G$ as the $1$-dimensional
skeleton. In other words, the simplices in ${\rm Cl}(G)$ consist of all
complete subgraphs in $G$.  
We define a clique complex process associated with the Erd\"os-R\'enyi graph process 
on $n$ vertices by 
\[
\CC(t) = {\rm Cl}(\K^{(1)}(t)), \quad 0 \le t \le 1, 
\]
where $\K^{(1)}(t)$ is the one defined in the previous subsection. 
The process starts from the $0$-skeleton, i.e., $n$ isolated vertices, and 
ends up with $\Delta_{n-1}$. Namely, 
\[
\Delta_{n-1}^{(0)} = \CC(0) \subset \CC(t)
\subset \CC(1) = \Delta_{n-1}. 
\]

By definition, for each edge $e$ in $\CC(t)$ (or equivalently
$\K^{(1)}(t)$), a uniform random variable $t_e \in
[0,1]$ is independently assigned as its birth time, and 
the birth time of a simplex $\sigma$ with $|\sigma| \ge 2$ is given by 
\[
 t_\sigma = \max\{t_e \mid e \subset \sigma,~|e|=2\}. 
\]
We remark that $t_v = 0$ for each vertex $v \in [n]$. 

Since a simplex $\sigma$ contains ${|\sigma| \choose 2}$ edges and $t_\sigma$ is the
maximum of the ordered statistics of i.i.d. ${|\sigma| \choose 2}$ uniform
random variables, we have 
\[
\E[t_\sigma] = \frac{{|\sigma| \choose 2}}{{|\sigma| \choose 2}+1}. 
\]
Here, we used the following well-known fact. Let $y_i, i=1,\dots, N,$ be 
i.i.d. uniform random variables on $[0,1]$ and $Y_i, i=1,\dots, N,$ be the rearrangement of $y_i$ in increasing order. 
Then, for each $i=1,\dots, N$, 
\begin{equation}
 \E[Y_i] = \frac{i}{N+1}. 
\label{eq:ordered}
\end{equation}

%\begin{rem}{\rm 
%While ${\beta}_k(t)$ in $\{\K^{(d)}(t)\}_{0\le t \le 1}$
% are monotone in $t$, 
%${\beta}_k(t)$ in $\{\CC(t)\}_{0\le t\le 1}$ are not. }
%\end{rem}

We remark that the Betti numbers in $\{\CC(t)\}_{0\le t\le 1}$
are not monotone in $t$, although they are in $\{\K^{(d)}(t)\}_{0\le t \le 1}$.

\section{Expectation of Lifetime Sum}\label{sec:main}
In this section, we first prove Theorem \ref{thm:main}. Then, we 
show a partial result on the expectation of the lifetime sum
in the clique complex process. 
We note that, since both processes $\{\K^{(d)}(t)\}_{0\le t \le 1}$ and $\{\CC(t)\}_{0\le t\le 1}$ are defined on the interval $[0,1]$, the lifetime formula (\ref{eq:int_intro2}) is given as
\begin{equation}\label{eq:int_proof}
L_{d-1}=\int_0^1\beta_{d-1}(t)dt.
\end{equation}

\subsection{Proof of Theorem \ref{thm:main}}
For $d\geq 1$, let $\cnd$ be the set of 
$d$-dimensional simplicial complexes on $n$ vertices with the $(d-1)$-complete skeleton 
$\Delta_{n-1}^{(d-1)}$. 
%Throughout this section, $N = {n \choose d+1}$. 
%We write $\tilde{\beta}_k(Y)$ for the reduced $k$-th Betti number of $Y$ over a field $\field$. 
%Since, if $Y \in \cnd$, it is identified with $Y_d$, we
%sometimes abuse the notation $Y$. 
For $Y \in \cnd$, let us define
\begin{align*}
 \reduce(Y) &= \{\sigma \in (\Delta_{n-1})_d \mid {\beta}_{d-1}(Y \cup \sigma) =
 {\beta}_{d-1}(Y) - 1\}, \\
 \shadow(Y) &= \{\sigma \in (\Delta_{n-1})_d 
\mid {\beta}_{d-1}(Y \cup \sigma) = {\beta}_{d-1}(Y)\}. 
\end{align*}
%where $(\Delta_{n-1})_d$ is the set of all $d$-faces. 
We note that 
\begin{enumerate}
 \item $Y_d \subset \shadow(Y)$,
 \item $(\Delta_{n-1})_{d} = \reduce(Y) \sqcup \shadow(Y)$ for 
$Y \in \cnd$, and 
 \item $\sigma \in \shadow(Y)$ is equivalent that the boundary of
       $\sigma$ is contained in $\image \partial_{Y,d}$,
%the linear span of the column vectors of $\partial_{Y,d}$. 
\end{enumerate}
where $\partial_{Y,d}$ is the $d$-th boundary map for $Y$.
%$Y_d \subset \shadow(Y)$ and 
%$(\Delta_{n-1})_{d} = \reduce(Y) \sqcup \shadow(Y)$ for 
%$Y \in \cnd$. 

The set $\shadow(Y) \setminus Y_d$ is called the shadow of $Y$ in \cite{LNPR}. 
It should be noted that $\reduce$ and $\shadow$ are monotone decreasing
and increasing, respectively, i.e.,
\[
\reduce(Y) \supset \reduce(Y'), \ \shadow(Y) \subset \shadow(Y')
%\label{decreasing}
\]
for $Y, Y' \in \cnd$ with $Y \subset Y'$.
For $Y \in \cnd$, we define the hull of $Y$ by 
$\overline{Y}:=Y \cup \shadow(Y)$. By definition, it is clear that 
\begin{equation}
{\beta}_{d-1}(\overline{Y}) = 
{\beta}_{d-1}(Y). 
\label{betahull}
\end{equation}

Now we use a Kruskal-Katona-type result obtained in \cite{LNPR}. 
Here, we restate their result as to be fitted in our situation. 
\begin{prop}[\cite{LNPR}, Corollary 6.6]
Let $Y$ be a $d$-dimensional simplicial complex with 
$|Y_d| = {x \choose d+1}$, where $x \ge d+1$ is a real. 
Then, $\rank \partial_{Y,d} \ge \frac{d+1}{x} |Y_d|$. 
In particular, for any $d$-dimensional simplicial complex $Y$ defined on $n$-vertices, 
\begin{equation}
\rank \partial_{Y,d} \ge \frac{d+1}{n} |Y_d|.  
\label{rank-cardinality}
\end{equation}
\end{prop}

\begin{cor}\label{cor:beta-cardinality}
For $Y \in \cnd$, 
\begin{equation}
 {\beta}_{d-1}(Y) \le \frac{d+1}{n} |\reduce(Y)|. 
\label{beta-cardinality}
\end{equation}
\end{cor}
\begin{proof} 
By \eqref{betahull} and \eqref{rank-cardinality}, 
\begin{align*}
 {\beta}_{d-1}(Y) 
= {\beta}_{d-1}(\overline{Y}) 
= {n-1 \choose d} - \rank \partial_{\overline{Y},d} 
\le {n-1 \choose d} - \frac{d+1}{n} |\overline{Y}_d|. 
\end{align*} 
Since $|\overline{Y}_d| = |\shadow(Y)| = {n \choose d+1} -
 |\reduce(Y)|$, 
we have the desired inequality. 
\end{proof}

%Given a family of random variables $X = (X_{\la})_{\la \in \La}$, 
%we call $X' = (X'_{\la})_{\la \in \La}$ a coupling of $X$ if 
%every $X'_{\la}$ 
%is realized on a common probability space
%$(\Omega, \f, \P)$ and $X'_{\la} \inlaw X_{\la}$ for each $\la \in \La$. 
%For two random simplicial complexes $X_1$ and $X_2$, 
%we write $X_1 \subset_{st} X_2$ if there exists a coupling 
%$X_i' \ (i=1,2)$ of $X_i \ (i=1,2)$ such that $X_1' \subset X_2'$.

In what follows, we will use the symbol $N$ for ${n \choose d+1}$ in this subsection. 
Let us set $\cnmd=\{Y \in \cnd\mid |Y_d| =m\}$. Then, we have a decomposition
\[
\cnd = \bigcup_{m=1}^{N} \cnmd. 
\]
%where $N={n\choose d+1}$.
Let $Y^{(d)}(n,m)$ be the uniform distribution on $\cnmd$. 
We use the notation $Y \sim Y^{(d)}(n,m)$ to mean that
$Y$ is chosen according to the distribution $Y^{(d)}(n,m)$.
%, i.e., $Y$ is uniformly chosen from $\cnmd$. 
%Let $Y^{(d)}(n,m)$ be the random simplicial complex uniformly chosen from
%$\cnmd$. 
%$d$-dimensional random complex on $n$ vertices with 
%the $(d-1)$-complete skeleton and uniformly randomly chosen $m$ $d$-faces. 

For two random simplicial complexes $X$ and $Y$ taking values in $\cnd$, 
we say that $Y$ stochastically dominates $X$, denoted by $X \subset_{st} Y$, if 
there exists a coupling of $X_d$ and $Y_d$ such that $X_d \subset Y_d$ a.s. 

%we write $X_1 \subset_{st} X_2$ if 
%$X_i' \ (i=1,2)$ of $X_i \ (i=1,2)$ such that $X_1' \subset X_2'$.

\begin{lem}\label{lem:coupling}
%Let $k, m \in \n$ with $km \le {n \choose d+1}$. 
Let $k, m \in \n$ with $km \le N$. 
Suppose that $Y_1,\dots, Y_k \sim Y^{(d)}(n,m)$ are i.i.d. random simplicial 
 complexes and $Y \sim Y^{(d)}(n,km)$. Then, 
$\cup_{i=1}^k Y_i  \subset_{st} Y$. 
%and hence $\reduce(Y) \subset_{st} \reduce(\cup_{i=1}^k Y_i)$. 
\end{lem}
\begin{proof}
%Let us define a collection of subsets of $d$-simplices 
%given $Y_1,\dots, Y_k$ by 
For given $Y_1,\dots, Y_k$, we define a collection of subsets of $d$-simplices by 
\[
 \a_{Y_1,\dots, Y_k} := \{F\subset (\Delta_{n-1})_d \mid 
F \supset \cup_{i=1}^k (Y_i)_d, \ |F| = km\}. 
\]
We sample $F$ from $\a_{Y_1,\dots, Y_k}$ uniformly at random 
and set $Y = \Delta_{n-1}^{(d-1)} \sqcup F$. 
Then, it is easy to see that the law of $Y$ is equal to $Y^{(d)}(n,km)$, and hence
$\cup_{i=1}^k Y_i  \subset_{st} Y$. 
%The second statement is derived by the monotone decreasing property of $\reduce$.
\end{proof}

For $Z \sim Y^{(d)}(n,m)$, we set $\rho_{n,m} = \P(\sigma \in \reduce(Z))$.
By symmetry, the probability $\rho_{n,m}$ does not 
depend on the choice of $\sigma \in (\Delta_{n-1})_d$, and 
$\E |\reduce(Z)| = N \rho_{n,m}$.
Note that $\rho_{n,m}$ is decreasing in $m$. 

\begin{lem}
Let $k, m \in \n$ with $km \le N$ and $Y \sim Y^{(d)}(n,km)$. 
Then, %for any $d$-simplex $\sigma$, 
\[
\E|\reduce(Y)| \le N \rho_{n,m}^k. 
%, 
%\P(\sigma \in \reduce(Y)) \le \rho_{n,m}^k, 
\] 
%where $\rho_{n,m} = \P(\sigma \in \reduce(Z))$ with 
%$Z \sim Y^{(d)}(n,m)$. 
\end{lem}
\begin{proof} 
Suppose $Y_1,\dots,Y_k \sim Y^{(d)}(n,m)$ are i.i.d. random simplicial complexes. 
From \lref{lem:coupling}, we have a coupling such that $Y_i \subset \cup_{i=1}^k Y_i \subset Y$ 
for every $i =1,2,\dots, k$ by symmetry. 
Since $\reduce$ is monotone decreasing, we obtain 
%%%%%It follows from monotonicity of $\reduce$ that 
\[
 \reduce(Y) 
\subset \reduce(\cup_{i=1}^k Y_i)
\subset \cap_{i=1}^k \reduce(Y_i).
\]
This implies 
\begin{align*}
\P(\sigma \in \reduce(Y)) 
&\le \P(\cap_{i=1}^k \{\sigma \in \reduce(Y_i)\}) \\
&= \P(\sigma \in \reduce(Y_1))^k \\ 
&=\rho_{n,m}^k.  
\end{align*}
%From the inclusion $\cup_{i=1}^k Y_i \subset_{st} Y$ and monotonicity of $\reduce$, 
%we have 
%\begin{align*}
%\P(\sigma \in \reduce(Y)) 
%&\le \P(\sigma \in \reduce(\cup_{i=1}^k Y_i)) \\
%&\le \P(\cap_{i=1}^k \{\sigma \in \reduce(Y_i)\}) \\
%&= \P(\sigma \in \reduce(Y_1))^k \\ 
%&=\rho_{n,m}^k.  
%\end{align*}
Therefore, again by symmetry, we obtain $\E|\reduce(Y)| \le N \rho_{n,m}^k$.  
\end{proof}

\begin{prop}\label{prop:bettiestimate} 
Let $\{Y_t=\cK^{(d)}(t)\}_{0 \le t \le 1}$ be the $d$-Linial-Meshulam process on $n$
 vertices. 
%Then, for any $m \le n$\note[H]{minor: $m\le N$?},  
Then, for any $m \le N$,  
\[
% \int_0^1 \E \tilde{\beta}_{d-1}(t) dt \le \frac{m}{1-\rho_{n,m}}. 
 \int_0^1 \E |\reduce(Y_t)| dt \le \frac{m}{1-\rho_{n,m}}. 
\]
\end{prop}
\begin{proof}
%Let $N = {n \choose d+1}$. 
For fixed $m \in \n$, we see that 
\begin{align*}
\E |\reduce(Y_t)| 
&= \sum_{k=0}^{\lfloor N/m\rfloor} \sum_{\ell=km}^{(k+1)m-1} 
\E[|\reduce(Y_t)| \ | \ |(Y_t)_d|=\ell] \cdot \P(|(Y_t)_d| = \ell) \\
&= \sum_{k=0}^{\lfloor N/m\rfloor} \sum_{\ell=km}^{(k+1)m-1} 
\E |\reduce(Y^{(d)}(n,\ell))| \cdot \P(|(Y_t)_d| = \ell) \\
&\le \sum_{k=0}^{\lfloor N/m\rfloor} \sum_{\ell=km}^{(k+1)m-1} 
\E |\reduce(Y^{(d)}(n, km))| \cdot \P(|(Y_t)_d| = \ell) \\
&\le \sum_{k=0}^{\lfloor N/m\rfloor} N \rho_{n,m}^k \sum_{\ell=km}^{(k+1)m-1} 
\P(|(Y_t)_d| = \ell).
\end{align*}
Here $\reduce(Y^{(d)}(n, \ell))$ means 
$\reduce(Y)$ for $Y \sim Y^{(d)}(n, \ell)$. 
Since $|(Y_t)_d| \sim Bin(N, t)$, we have 
\[
 \int_0^1 \P(|(Y_t)_d| = \ell) dt 
= \int_0^1 {N \choose \ell} t^\ell (1-t)^{N-\ell} dt 
%= {N \choose l} B(l+1, N-l+1) 
= \frac{1}{N+1}.
\]
Therefore, 
\[
\int_0^1 \E |\reduce(Y_t)| dt 
\le \sum_{k=0}^{\lfloor N/m\rfloor} N \rho_{n,m}^k \frac{m}{N+1} 
\le \frac{m}{1-\rho_{n,m}}. 
\]
\end{proof}

Now, we appropriately choose $m$ in Proposition \ref{prop:bettiestimate}.  
For $0 < c < 1$, let us define
\begin{equation}
 m_c(n) := \min\left\{ m \le N~\Big{|}~ \rho_{n,m} \le c
 \right\}. 
\label{eq:mcn}
\end{equation}
Then, Hoffman-Kahle-Paquette showed the following result.
\begin{lem}[\cite{HKP2}, Lemma 10]\label{lem:hkp}
$m_{1/2}(n) \le 4 {n \choose d}$. 
\end{lem}
%
%In \cite{HKP2}, they define 
%$m_{1/2}(n)= \min \{ m \le N ~ | ~ \E|\reduce(Y)| \le N/2\}$, which 
%is equivalent to \eqref{eq:mcn}. 

%We remark that $m_{1/2}(n)= \min \{ m \le N ~ | ~ \E|\reduce(Y)| \le N/2\}$ is used in \cite{HKP2} in a slightly different way, although it is equivalent to \eqref{eq:mcn} with $c=1/2$ by symmetry. 
We remark that a slightly different definition $m_{1/2}(n)= \min \{ m \le N ~ | ~ \E|\reduce(Y)| \le N/2\}$ is used in \cite{HKP2}, and it is equivalent to \eqref{eq:mcn} with $c=1/2$ by symmetry. 

\begin{prop}\label{integralreduce}
Let $\{Y_t\}_{0 \le t \le 1}$ be the $d$-Linial-Meshulam process on $n$
 vertices. Then, 
\[
 \int_0^1 \E |\reduce(Y_t)| dt \le 8 {n \choose d}. 
\]
\end{prop}
\begin{proof}
We set $m = m_{1/2}(n)$ in Proposition \ref{prop:bettiestimate}. 
%Then, $\rho_{n,m} \le 1/2$ from \lref{lem:mcestimate} and 
Then, $\rho_{n,m} \le 1/2$ from \eqref{eq:mcn} and hence 
\[
 \int_0^1 \E |\reduce(Y_t)| dt 
\le \frac{m}{1-\rho_{n,m}} 
\le 2 m_{1/2}(n) \le 8 {n \choose d}. 
\]
\end{proof} 

Now we are in a position to prove our main result.
\vspace{0.3cm}\\
{\it Proof of Theorem \ref{thm:main}.}
By the lifetime formula (\ref{eq:int_proof}), 
Corollary \ref{cor:beta-cardinality} and \pref{integralreduce}, 
we obtain an upper bound
\[
\E[L_{d-1}]  =
\int_0^1 \E [{\beta}_{d-1}(t)] dt 
\le \frac{d+1}{n} \int_0^1 \E |\reduce(Y_t)| dt 
\le 8 \frac{d+1}{n} {n \choose d}
\sim \frac{8(d+1)}{d!} n^{d-1}. 
\]

Let us next consider a lower bound. 
Because of $t_\eta = 0$ for any $(d-1)$-simplex $\eta$ in the $d$-Linial-Meshulam process, the lifetime formula (\ref{eq:int_intro1}) leads to
\[
 L_{d-1} = \wt(T) \ge \sum_{i=1}^{|T|} u_i, 
\]
where $T$ is the minimum spanning $d$-acycle and 
$0 \le u_1 \le u_2 \le \dots \le u_N \le 1$ 
%$0 \le u_1 \le u_2 \le \dots \le u_{{n \choose d+1}} \le 1$ 
is the rearrangement of i.i.d. uniform random variables $\{t_\sigma\mid
 \sigma \in (\Delta_{n-1})_d \}$.  
We recall $|T| = {n-1 \choose d}$ from Example \ref{ex:cardinality}. 
Then, it follows from (\ref{eq:ordered}) that  
\begin{align*}
\E[L_{d-1}] 
&\ge \sum_{i=1}^{|T|} \E[u_i]  
%= \sum_{i=1}^{|S|} \frac{i}{{n \choose d+1}+1} 
= \sum_{i=1}^{|T|} \frac{i}{N+1} 
\sim \frac{d+1}{2d!} n^{d-1}. 
%= O(n^{d-1}). 
%\\
%&= \frac{1}{{n \choose d+1}+1} \frac{|S|(|S|+1)}{2} \\
%&= O(n^{d-1}). 
\end{align*}
This completes the proof.
\qed

\newcommand{\lex}{<_{lex}}
\newcommand{\tP}{\tilde{P}}

\subsection{Expected Lifetime Sum for Clique Complex Process}
In this subsection, we consider the clique complex process and show bounds for the expectation of the lifetime sum. 
For the derivation of the upper bound, we first recall the discrete Morse theory. 
\begin{defn}{\rm 
	Let $X$ be a simplicial complex on a vertex set $V$. A partial matching
	consists of a partition of $X$ into three sets $A,
	Q$, and $K$ along with a bijection $\phi:Q\rightarrow K$ such that
	$\sigma\subset \phi(\sigma)$ and $|\phi(\sigma)|=|\sigma|+1$ for each
	$\sigma\in Q$.  }
\end{defn}

We denote a partial matching by $\mathcal{M}=(A,\phi:Q\rightarrow K)$. 
Given a partial matching, we set a relation $\ll$ on $Q$ by extending 
transitively the relation $\triangleleft$ defined by 
\[
	Q'\triangleleft Q \Longleftrightarrow Q' \subset \phi(Q).
\]
A partial matching $\mathcal{M}$ is called \textit{an acyclic matching} 
if $\ll$ is a partial order.
The elements in $A$ are called \textit{critical} simplices.

\begin{thm}[\cite{forman}]\label{homotopic2CW}
Suppose $X$ is a simplicial complex with an acyclic matching $\mathcal{M}$. 
Then, $X$ is homotopy equivalent to a CW complex with exactly one $k$-cell for 
each critical $k$-simplex. 
\end{thm}
%\note[S]{Reference}

%A closed $\mathcal{M}$-path is a sequence of simplices 
%\[
% \sigma_0\subset\tau_0 \supset
% \sigma_1\subset\tau_1\supset\sigma_2\subset\dots 
% \supset\sigma_m\subset\tau_{m}\supset\sigma_0   
%\]
%such that $\tau_i=\phi(\sigma_i)$ for $i=0,\dots,m$. 
%We note $m\geq 2$, since $m=1$ implies
%$\phi(\sigma_0)=\tau_0=\tau_1=\phi(\sigma_1)$. 

\def\lex{<_{lex}}
Let us construct a $(d-1,d)$-type acyclic matching 
$(A,\phi:Q\rightarrow K)$ as follows: 
\begin{enumerate}
\item Suppose that the vertex set $V$ is totally ordered as $1 < 2 < \dots <
|V|$ and it induces the lexicographic order $\lex$ on $X$. 
%We write $\sigma \lex \tau$ 
%if $\sigma$ comes before $\tau$ in the lexicographic order.
\item For $\sigma\in X_{d-1}$, if there exists
\[
	\tau = \lexmin\{\tilde{\tau} \in X_d \mid 
      \sigma \subset \tilde{\tau}, \sigma \lex \tilde{\tau}\}, 
\]
then we add $\phi:\sigma \mapsto \tau$ as a pairing.
\item All the remaining simplices are set to be critical. 
\end{enumerate}
This is the acyclic matching used in {\rm \cite{kahle}}. 

Now we derive the following bounds of the expected lifetime sum in the
clique complex process.  

\begin{thm}\label{thm:clique}
For the clique complex process $\{\CC(t)\}_{0 \le t \le 1}$, 
there exist positive constants $c$ and $C$ (depending on $d$) 
such that, as $n \to \infty$, 
%\[
%c n^{\frac{(d+2)(d-1)}{2d}} \le \E[L_{d-1}] \le C n^{d-1}
%\]
%for $d \ge 3$ and 
%\[
%c n^{d-1} \le \E[L_{d-1}] \le C n^{d-1} \log n
%\]
%for $d=1,2$.  
\[
c n^{d-1} \le \E[L_{d-1}] \le C n^{d-1} \log n
\]
for $d=1,2$ and   
\[
c n^{\frac{(d+2)(d-1)}{2d}} \le \E[L_{d-1}] \le C n^{d-1}
\]
for $d \ge 3$.
\end{thm}
We remark that $d-1=\frac{(d+2)(d-1)}{2d}$ for $d=1,2$.
%\begin{rem} (need to be checked)
%The upper bound for $d \ge 3$ might be improved and 
%the statement would be of the following simpler form:  
%\[
% c n^{\frac{(d+2)(d-1)}{2d}} \le \E[L_{d-1}] \le C
% n^{\frac{(d+2)(d-1)}{2d}} (\log n)^{\frac{d-1}{2} + \frac{1}{d}} 
%\] 
%for $d \ge 1$. 
%\end{rem}
\begin{proof} 
Let us write $f_i(t)=f_i(\cC(t))$ and $\beta_i(t)=\beta_i(\cC(t))$.
We recall the Morse inequality 
\[
\sum_{j=d-2}^d (-1)^{d-1-j} f_j(t) 
\le {\beta}_{d-1}(t). 
%\le f_{d-1}(\CC(t))
\]
We observe that 
\begin{align*}
\int_0^{\left(\frac{d+1}{n}\right)^{1/d}}
\E[ f_j(t) ] dt 
%\int_0^{O(n^{-\frac{1}{d}})}  \E[ f_j(\CC(t)) ] dt
&= 
%\int_0^{O(n^{-\frac{1}{d}})} 
\int_0^{\left(\frac{d+1}{n}\right)^{1/d}}
{n \choose j+1} t^{{j+1 \choose 2}} dt \\
%&= O\Big(n^{j+1 - \frac{1}{d}\{{j+1 \choose 2} + 1\}}
%\Big)
&\sim A_j^{(d)} n^{j+1 - \frac{1}{d}\{{j+1 \choose 2} + 1\}}
\\
%&= 
&=
\begin{cases}
A_j^{(d)} n^{\frac{(d+2)(d-1)}{2d}}, & j=d-1,d.\\ 
A_j^{(d)} n^{\frac{(d+2)(d-1)-2}{2d}}, & j=d-2. 
%O(n^{\frac{(d+2)(d-1)}{2d}}), & j=d-1,d.\\ 
%O(n^{\frac{(d+2)(d-1)-2}{2d}}), & j=d-2. 
\end{cases}
\end{align*}
It is easy to check that $c_{d-1} := A_{d-1}^{(d)} - A_{d}^{(d)}>0$ for
 every $d \ge 1$. 
Therefore, by (\ref{eq:int_proof}) and the Morse inequality, we see that 
\begin{align*}
 \E[ L_{d-1} ]
&\ge 
\int_0^{\left(\frac{d+1}{n}\right)^{1/d}}
%\int_0^{O(n^{-\frac{1}{d}})}  
\E[{\beta}_{d-1}(t)]dt \\
&\gtrsim c_{d-1} n^{\frac{(d+2)(d-1)}{2d}}. 
%&= O(n^{\frac{(d+2)(d-1)}{2d}}). 
\end{align*}
This yields the lower bound since 
$\frac{(d+2)(d-1)}{2d} = d-1$ for $d=1,2$. 

For the upper bound, we use the discrete Morse theory. 
Let $f_{d-1}^*(t)$ be the number of critical $(d-1)$-simplices 
in the $(d-1, d)$-type acyclic matching for $\CC(t)$. 
It follows from Theorem~\ref{homotopic2CW} that
\[
 {\beta}_{d-1}(t) \le \min\{f_{d-1}(t), f_{d-1}^*(t)\}. 
\]
The expectation of the first term is given by $\E[f_{d-1}(t)]={n\choose d}t^{{d\choose 2}}$. On the other hand, we compute $\E[f_{d-1}^*(t)]$ as 
\begin{align*}
\E[f_{d-1}^*(t)]
&= \sum_{\sigma \in (\Delta_{n-1})_{d-1}} \P(\text{$\sigma$ is
critical}) \\
&= \sum_{j=d}^n \sum_{1 \le i_1< i_2 < \dots < i_{d-1} < j} 
\P(\text{$\{i_1,i_2, \dots, i_{d-1},j\}$ is critical}). 
\end{align*}
A $(d-1)$-simplex $\{i_1,i_2, \dots, i_{d-1},j\}$ is critical if and only if 
the $d$-simplex $\{i_1,i_2, \dots, i_{d-1},j,k\}$ does not appear for any $k \ge j+1$. 
Hence, we obtain 
\begin{eqnarray*}
\E[f_{d-1}^*(t)]
&=& \sum_{j=d}^n {j-1 \choose d-1} t^{{d\choose 2}}(1-t^d)^{n-j}\\
&\le&{n\choose d-1}t^{{d\choose 2}}\sum_{j=d}^n(1-t^d)^{n-j}\\
&=&{n\choose d-1}t^{{d\choose 2}-d},
\end{eqnarray*}
and
\[
 \E[{\beta}_{d-1}(t)] \le \min
\left\{
{n \choose d} t^{{d \choose 2}}, \ 
{n \choose d-1} t^{{d \choose 2} - d}
\right\}. 
\]

Therefore, we obtain
\[
\E[L_{d-1}] \le \int_0^1 {n \choose d-1} t^{{d \choose 2} - d} dt 
= O(n^{d-1}) 
\]
for $d \ge 3$ and
\begin{align*}
\E[L_{d-1}] 
& \le 
%\int_0^{a_n} {n \choose d} t^{{d \choose 2}} dt 
%+ \int_{a_n}^1 {n \choose d-1} t^{{d \choose 2} - d} dt \\
\int_0^{n^{-1/d}} {n \choose d} t^{{d \choose 2}} dt 
+ \int_{n^{-1/d}}^1 {n \choose d-1} t^{{d \choose 2} - d} dt \\
&= O(n^{\frac{(d-1)(d+2)}{2d}}) + O(n^{d-1} \log n) \\
&= O(n^{d-1} \log n) 
\end{align*}
for $d=1,2$. 
This completes the proof. 
\end{proof}

\section{Concluding Remarks}\label{sec:conclusion}
\subsection{Limiting Constant}
A detailed analysis of the Linial-Meshulam complex $\cK^{(d)}(t)$ at 
time $t = c/n \ ( c \ge 0)$ has recently been reported in \cite{lp}. 
By applying their results, we formally show that the limit 
$I_{d-1}:=\lim_{n\rightarrow \infty}\frac{1}{n^{d-1}}\E[L_{d-1}]$ can be expressed 
by an integral form which recovers $I_0=\zeta(3)$ for $d=1$. 

%As was studied in the Erd\"os-R\'enyi graph, each ``connected component'' in the complex can be well approximated by Poisson trees at $t=c/n$. 
As was studied in the Erd\"os-R\'enyi graph, Poisson trees play an important role to characterize the Linial-Meshulam complex at $t=c/n$.
By using the spectral measure of the upper $(d-1)$-dimensional 
Laplacian $\partial_{d} \partial_{d}^T$ 
obtained from the boundary operator, Linial-Peled \cite{lp} basically show the following:  
let $t_d^*$ be the unique root in $(0,1)$ of the following equation
\[
 (d+1) (1-t) + (1+dt) \log t = 0 
\]
and set $c_d^* = \psi_d(t_d^*)$ by
\[
 \psi_d(t) = \frac{- \log t}{(1-t)^d}, \quad  t \in (0,1). 
\]
For $d=1$, we understand $t_1^* = c_1^* = 1$. 
%\[
% c_d^* := \frac{- \log t_d^*}{(1-t_d^*)^d}. 
%\]
Let $t = t_c$ be the smallest positive root of the equation $t =
e^{-c(1-t)^d}$. 
%Remark that $t=1$ is always a root. 
Then, for every $c > c_d^*$, 
\[
%\frac{1}{{n \choose d}} \E[\dim H_d(X(\frac{c}{n}))] 
\frac{1}{{n \choose d}} \E[\beta_d(c/n)] 
= (1+ o(1)) 
\{c t_c (1-t_c)^d + \frac{c}{d+1}(1-t_c)^{d+1} -
(1-t_c)\}
\]
holds with probability tending to $1$ as $n \to \infty$. 

Now, by applying this asymptotic formula into \eqref{lmd}, we have
\begin{align*}
\frac{1}{{n \choose d}} 
\E [{\beta}_{d-1}(c/n)]
&= \frac{1}{{n \choose d}} 
\left\{\E [{\beta}_d(c/n)] + {n-1 \choose d} -
 \E [f_d(c/n)] \right\} \\
&= (1+o(1)) 
\underbrace{\left\{c t_c (1-t_c)^d + \frac{c}{d+1}(1-t_c)^{d+1} +
 t_c - \frac{c}{d+1}\right\}}_{=: h_d(c)}. 
%&=: h_d(c) 
\end{align*}
Then, integrating both sides roughly yields
\begin{align*}
\frac{1}{n^{d-1}} \int_0^1 \E [{\beta}_{d-1}(t)] dt 
&\sim \frac{1}{d!} {n \choose d}^{-1} 
\int_0^n \E [{\beta}_{d-1}(c/n)] dc
\approx \frac{1}{d!} \int_0^{\infty} h_d(c) dc
\end{align*}
as $n\rightarrow \infty$.
We should remark that the part $\approx$ is the rough derivation. 
From this formal discussion, we conjecture that the limit exists and the limiting constant is given by 
\[
I_{d-1}=\lim_{n\rightarrow \infty}\frac{1}{n^{d-1}}\E[L_{d-1}] 
%= \lim_{n \to \infty} \frac{1}{n^{d-1}} 
%\int_0^1 \E
%{\beta}_{d-1}(t) dt
= \frac{1}{d!} \int_0^{\infty} h_d(c) dc.
\]

Actually, this integral form recovers $I_0=\zeta(3)$, which is nothing but Frieze's $\zeta(3)$-limit theorem. Namely, for $d=1$, 
%then $t_1^*=c_1^* = 1$; 
we have $t_c=1$ for $0 \le c \le 1$ and 
$t_c=\psi_1^{-1}(c)$ for $c \ge 1$. 
Then, we obtain
\begin{align*}
 I_0
&= \int_0^{1} (1 - \frac{c}{2}) dc + 
 \int_{1}^{\infty} h_1(c) dc \\
&= \frac{3}{4} + \int_0^1 
\frac{(2-2t + t \log t)(1- t + t \log t)}{2(1-t)^3} dt \\
&=\zeta(3). 
\end{align*}
Here, the second equality follows 
by the change of variables $c = \psi_1(t)$ for $1 \le c < \infty$. 
%This recovers Frieze's $\zeta(3)$-limit theorem. 

\subsection{Central Limit Theorem}
After Frieze's work, Janson \cite{janson} proved the central limit theorem
\[
	\sqrt{n} (L_{0} - \zeta(3)) \stackrel{d}{\Longrightarrow} N(0,\sigma^2)
\]
with $\sigma^2 = 6\zeta(4) - 4\zeta(3)$. 
Hence, the next interesting problem is to find the variance of $L_{d-1}$, and furthermore to establish 
the (functional) central limit theorem in our setting.
To this aim, we feel that we need more detailed study of structures of spanning acycles.

\subsection{Limit Theorem of Persistence Diagram}
In Section 5, we viewed random persistence diagrams as point processes on $\Delta$. 
Then, the lifetime sum $L_{d-1}$ is a functional of this point process and we derived its order in this paper. 
So, another natural problem is to establish the limit theorem for persistence diagrams.
Moreover, limit theorems for barcodes and persistent landscapes for the Linial-Meshulam processes and the clique complex process 
should also be studied in connection with persistence diagrams.

\subsection{Order in the Clique Complex Process}
Theorem \ref{thm:clique} shows the upper and lower bounds of $\E[L_{d-1}]$ in the clique complex process, but the explicit order has not yet been obtained at present. 
We performed numerical experiments to observe $\E[L_{d-1}]$ with respect to the number of vertices. From these computations, we observe that the upper bound seems to be correct for $d=2$, although
the lower bound is the right order for $d=1$ (Frieze's $\zeta(3)$-limit theorem).

\subsection{Asymptotics of \mbox{\boldmath$\ell^2$}-norm}
In Remark \ref{rem:landscape}, we showed the integral formula for the $\ell^2$-norm $\|\vec{l}\|_2$ of the sequence $\vec{l} = (l_i)_{i=1}^p$ of the lifetimes in connection with the persistence landscape. 
Then, it seems to be interesting to study asymptotic behaviors of $\|\vec{l}\|_2$ in a similar spirit to our main result. For this purpose, we would like to also derive an algebraic formulation of $\|\vec{l}\|_2$ corresponding to (\ref{eq:int_intro1}). 

\subsection{Wilson's Algorithm}
The Wilson's algorithm \cite{wilson} provides a fast algorithm of 
sampling uniform spanning trees by using loop-erased random walks on graphs. 
It would be natural to ask a generalization of the Wilson's algorithm 
producing uniform spanning acycles. 
To this aim, there are two things we need to consider. 
One is to find a natural candidate of loop-erased random walks defined 
on simplicial complexes. 
The other  is to give a right meaning of \textit{uniform} 
under the presence of the nontrivial factor (\ref{eq:det_homology})
in (\ref{eq:det_expansion}). In case of graphs ($d=1$), since this factor is always $1$, the weight of each spanning tree is not biased.

\section*{Acknowledgement}
This work is partially supported by JSPS Grant-in-Aid (26610025, 26287019).


\begin{thebibliography}{99}
%\bibitem{spencer}
%	N. Alon and J. Spencer. The Probabilistic Method. 2nd edition, John Wiley, 2000.

%\bibitem{BS}
%	W.~Ballman and J.~\'Swi\c{a}tkowski. 
%	On $L^2$-cohomology and property (T) for automorphism groups of 	polyhedral cell complexes. 
%	{\it Geom.~Funt.~Anal.} 7 (1997), 615--645.

\bibitem{B}
	P.~Bubenik. 
	Statistical topological data analysis using persistent diagram.  
	arXiv:1207.6437v4.

\bibitem{cfijs}
	C.~Cooper, A.~Frieze, N.~Ince, S.~Janson and J.~Spencer. 
	On the length of a random minimum spanning tree. 
	arXiv:1208.5170v2.

%\bibitem{cb}
%	W. Crawley-Boevey. 
%	Decomposition of pointwise finite-dimensional persistence modules. 
%	arXiv:1210.0819v3. 

\bibitem{matrix_tree_thm}
	A. Duval, C. J. Klivans and J. L. Martin. 
	Simplicial matrix-tree theorems. 
	{\it Trans. Amer. Math. Soc.} 361 (2009), 6073--6114.

\bibitem{elz}
	H. Edelsbrunner, D. Letscher, and A. Zomorodian. 
	Topological Persistence and Simplification. 
	{\it Discrete Comput. Geom.} 28 (2002), 511--533.

\bibitem{er}
	P. Erd{\"o}s and A. R{\'e}nyi. 
	On random graphs I. 
	{\it Publ. Math. Debrecen} 6 (1959), 290--297.

\bibitem{forman} 
	R.~Forman. 
	Morse theory for cell complexes. 
	{\it Adv. Math.} 134 (1998), 90--145.

\bibitem{frieze}
	A. M. Frieze. 
	On the value of a random minimum spanning tree problem. 
	{\it Discrete Applied Math.} 10 (1985), 47--56.

\bibitem{munkres}
	J. R. Munkres. Elements of Algebraic Topology.
	Perseus Publishing. 1984.

\bibitem{HKP}
	C.~Hoffman, M.~Kahle and E. Paquette. 
	Spectral gaps of random graphs and applications to random topology. 
	{\it Discrete Math.} 309 (2009), 1658--1671.

\bibitem{HKP2}
	C.~Hoffman, M.~Kahle and E. Paquette. 
	The threshold for integer homology in random $d$-complexes. 
	arXiv:1308.6232v2.

\bibitem{janson}
	S. Janson. The Minimal spanning tree in a complete graph and a functional limit theorem for trees in a random graph. 
	{\it Random Struct. Alg.} 7 (1995), 337--355.

\bibitem{kahle}
	M. Kahle. 
	Topology of random clique complexes. 
	{\it Discrete Math.} 309 (2009), 1658--1671.

\bibitem{kahle2}
	M. Kahle. 
	Topology of random simplicial complexes: a survey. 
	arXiv:1301.7165v2.

\bibitem{kalai}
	G. Kalai. 
	Enumeration of $Q$-acyclic simplicial complexes. 
	{\it Israel J. Math.} 45 (1983), 337--351.

\bibitem{kruskal}
	J.~B.~Kruskal. 
	On the shortest spanning subtree of a graph and
	the traveling salesman problem. 
	{\it Proc. Amer. Math. Soc.} 7 (1956), 48--50. 

\bibitem{lm}
	N.~Linial and R.~Meshulam. 
	Homological connectivity of random $2$-complexes. 
	{\it Combinatorica} 26 (2006), 475--487. 

\bibitem{LNPR} 
	N.~Linial, I.~Newman, Y.~Peled and Y.~Rabinovich. 
	Extremal problems on shadows and hypercuts in simplicial complexes. 
	arXiv:1408.0602v2.
	
\bibitem{lp}
	N.~Linial and Y.~Peled. 
	On the phase transition in random simplicial complexes. 
	arXiv:1410.1281.

\bibitem{little}
	J.~Little. 
	A proof for the queuing formula: $L=\lambda W$.
	{\it Operations Research} 9 (1961), 383--387.

\bibitem{lyons} 
	R.~Lyons. 
	Random complexes and $\ell^2$-Betti numbers. 
	{\it J. Topology Anal.} 1 (2009), 153--175. 

\bibitem{mw}
	R.~Meshulam and N.~Wallach. 
	Homological connectivity of random $k$-dimensional complexes. 
	{\it Random Struct. Alg.} 34 (2009), 408--417. 

\bibitem{wilson} 
	D.~B.~Wilson. 
	Generating random spanning trees more quickly than the cover
	time. 
	Proceedings of the Twenty-eighth Annual ACM Symposium on the
	Theory of Computing (Philadelphia, PA, 1996), 296–303, ACM, New York,
	1996. 
	
\bibitem{zc}
	A. Zomorodian and G. Carlsson. 
	Computing persistent homology. 
	{\it Discrete Comput. Geom.} 33 (2005), 249-274.
	
\end{thebibliography}
\end{document}